\documentclass{article}
\usepackage{miyamath, amsthm, ./main}
\usepackage{soul}
\usepackage{tikz} \usetikzlibrary{matrix,arrows}

\newcommand{\pdag}[1]{{}^{\dag}\! #1}
\mathchardef\hyphen="2D
\renewcommand{\inv}{\mathrm{inv}}

\newcommand{\totimes}{\widetilde{\otimes}}
\DeclarePairedDelimiter\floor{\lfloor}{\rfloor}
\newtheorem*{cvthm}{Convolution Theorem over {$\bC$}}
\newtheorem*{mainthm}{Main Theorem}
\title{$p$-adic Generalized Hypergeometric Equations from the Viewpoint of Arithmetic $\sD$-modules.}
\author{Kazuaki Miyatani}
\begin{document}
\maketitle

\begin{abstract}
    We study the $p$-adic (generalized) hypergeometric equations
    by using the theory of multiplicative convolution of arithmetic $\sD$-modules.
    As a result, we prove that the hypergeometric isocrystals with suitable rational parameters
    have a structure of overconvergent $F$-isocrystals.
\end{abstract}

\section{Introduction.}

Katz \cite[5.3.1]{Katz90ESDE} discovered that the hypergeometric $\sD$-modules on $\bA^1_{\bC}\setminus\{0\}$
can be described as the multiplicative convolution of hypergeometric $\sD$-modules of rank one.
Precisely speaking, Katz proved the statement (ii) in the following theorem (the statement (i) is trivial
but put to compare with another theorem later).

\begin{cvthm}
Let $\balpha=(\alpha_1,\dots,\alpha_m)$ and $\bbeta=(\beta_1,\dots,\beta_n)$ be two sequences of complex numbers
and assume that $\alpha_i-\beta_j$ is not an integer for any $i,j$.
Let $\sHyp(\balpha;\bbeta)$ be the $\sD$-module on $\bG_{\rmm,\bC}$ defined by the 
hypergeometric operator
\[
    \Hyp(\balpha;\bbeta)=\prod_{i=1}^m(x\partial-\alpha_i)-x\prod_{j=1}^n(x\partial-\beta_j),
\]
that is,
\[
    \sHyp(\balpha;\bbeta)\defeq\sD_{\bA^1_{\bC}\setminus\{0\}}/\sD_{\bA^1_{\bC}\setminus\{0\}}\Hyp(\balpha;\bbeta).
\]
Then, $\sHyp(\balpha;\bbeta)$ has the following properties.

\textup{(i)} If $m\neq n$, then $\sHyp(\balpha;\bbeta)$ is a free $\sO_{\bG_{\rmm,\bC}}$-module of rank $\max\{m,n\}$.
If $m=n$, then the restriction of $\sHyp(\balpha;\bbeta)$ to $\bG_{\rmm,\bC}\setminus\{1\}$ is a free $\sO_{\bG_{\rmm,\bC}\setminus\{1\}}$-module of rank $m$.

\textup{(ii)} We have an isomorphism
\[
    \sHyp(\balpha;\bbeta)\cong \sHyp(\alpha_1;\emptyset)\ast\dots\ast\sHyp(\alpha_m;\emptyset)
    \ast\sHyp(\emptyset;\beta_1)\ast\dots\ast\sHyp(\emptyset;\beta_n),
\]
where $\ast$ denotes the multiplicative convolution of $\sD_{\bG_{\rmm,\bC}}$-modules.
\end{cvthm}

Besides the hypergeometric $\sD$-modules over the complex numbers,
Katz also studied the $\ell$-adic theory of hypergeometric sheaves.
Let $k$ be a finite field with $q$ elements,
let $\psi$ be a non-trivial additive character on $k$ and let $\bchi=(\chi_1,\dots,\chi_m), \brho=(\rho_1,\dots,\rho_n)$
be sequences of characters on $k^{\times}$ satisfying $\chi_i\neq\rho_j$ for all $i,j$.
Then, he \emph{defined} the $\ell$-adic hypergeometric sheaves $\sH_{\psi,!}^{\ell}(\bchi,\brho)$ on $\bG_{\rmm,k}$
by using the multiplicative convolution
of $\sH_{\psi,!}^{\ell}(\chi_i;\emptyset)$'s and $\sHyp_{\psi,!}^{\ell}(\emptyset;\rho_j)$'s,
where these convolvends are defined by using Artin--Schreier sheaves and Kummer sheaves.

This $\ell$-adic sheaf $\sH_{\psi,!}^{\ell}(\bchi;\brho)$ has a property similar to (i) in the above theorem.
Namely, it is a smooth sheaf on $\bG_{\rmm,k}$ of rank $\max\{m,n\}$ if $m\neq n$,
and its restriction to $\bG_{\rmm,k}\setminus\{1\}$ is a smooth sheaf of rank $m$ if $m=n$
\cite[Theorem 8.4.2]{Katz90ESDE}.

Moreover, by definition, $\sH_{\psi,!}^{\ell}(\bchi;\brho)$ has a Frobenius structure.
The Frobenius trace functions of the $\ell$-adic hypergeometric sheaves are called the ``hypergeometric functions over finite field''.
This function gives a generalization of the classical Kloosterman sums.
Moreover, this function has an intimate connection with the Frobenius action on the \'etale cohomology of 
a certain class of algebraic varieties (for example, Calabi--Yau varieties)
over finite fields.
(The hypergeometric function over finite field is also called the ``Gaussian hypergeometric function'' by Greene \cite{Greene},
who independently of Katz found this function based on a different motivation.)

The purpose of this article is to develop a $p$-adic counterpart of these complex and $\ell$-adic hypergeometric objects.
This $p$-adic hypergeometric object will have a presentation in terms of the ($p$-adic) differential equation,
and at the same time has a Frobenius structure.
For its formalisation, we exploit the theory of arithmetic $\sD$-modules introduced by Berthelot.
The main theorem of this article is stated as follows.

\begin{mainthm}
    Let $K$ be a complete discrete valuation field of mixed characteristic $(0,p)$ with residue field $k$,
    a finite field with $q$ elements.
    Let $\pi$ be an element of $K$ that satisfies $\pi^{q-1}=(-p)^{(q-1)/(p-1)}$.
    Let $\balpha=(\alpha_1,\dots,\alpha_m)$ and $\bbeta=(\beta_1,\dots,\beta_n)$ be 
    two sequences of elements of $\frac{1}{q-1}\bZ$,
    and assume that $\alpha_i-\beta_j$ is not an integer for any $i,j$.
    Let $\sHyp_{\pi}(\balpha;\bbeta)$ be the $\sD^{\dag}_{\widehat{\bP}^1,\bQ}(\pdag{\{0,\infty\}})$-module defined by
    the $p$-adic hypergeometric differential operator
    \[
        \Hyp_{\pi}(\balpha;\bbeta)=\prod_{i=1}^m(x\partial-\alpha_i)-(-1)^{m+np}\pi^{m-n}x\prod_{j=1}^n(x\partial -\beta_j),
    \]
    that is,
    \[
        \sHyp_{\pi}(\balpha;\bbeta)\defeq\sD^{\dag}_{\widehat{\bP^1_V},\bQ}(\pdag{\{0,\infty\}})/
        \sD^{\dag}_{\widehat{\bP^1_V},\bQ}(\pdag{\{0,\infty\}})\Hyp_{\pi}(\balpha;\bbeta).
    \]
    Then, $\sHyp_{\pi}(\balpha;\bbeta)$ has the following properties.

    \textup{(i)} If $m\neq n$, then $\sHyp_{\pi}(\balpha;\bbeta)$ has a structure of an overconvergent $F$-isocrystal on $\bG_{\rmm,k}$
    of rank $\max{\{m,n\}}$.
    If $m=n$, then the restriction of $\sHyp_{\pi}(\balpha;\bbeta)$ to $\widehat{\bP^1_V}\setminus\{0,1,\infty\}$ has a structure of an overconvergent $F$-isocrystal on $\bG_{\rmm, k}\setminus\{1\}$ of rank $m$.

    \textup{(ii)} We have an isomorphism
    \[
        \sHyp_{\pi}(\balpha;\bbeta)\cong\sHyp_{\pi}(\alpha_i;\emptyset)\ast\dots\ast\sHyp_{\pi}(\alpha_m;\emptyset)\ast
        \sHyp_{\pi}(\emptyset;\beta_1)\ast\dots\ast\sHyp_{\pi}(\emptyset;\beta_n).
    \]
    
    \textup{(iii)} Let $x$ be a closed point in $\bG_{\rmm,k}$ if $m\neq n$, and in $\bG_{\rmm,k}\setminus\{1\}$ of $m=n$.
    Then, the Frobenius trace of the overconvergent $F$-isocrystal obtained in (i) at $x$ equals that of
    the $\ell$-adic hypergeometric sheaf $\sH_{\psi,!}^{\ell}(\bchi;\brho)$,
    where $\psi$, $\bchi$ and $\brho$ are the corresponding (sequence of) characters to $\pi$, $\balpha$ and $\bbeta$ respectively.
\end{mainthm}

In this theorem, (i) states not only that it is a free $\sO$-module of the desired rank, but also that this $\sD$-module satisfies a kind of convergence condition and that it has a Frobenius structure.
Despite the triviality of (i) of Convolution Theorem over $\bC$, this is therefore a highly non-trivial statement.

Another difference from the over-$\bC$ case is that we restricted the parameters to rational coefficients.
We expect that, even for other $p$-adic parameters not necessarily coming from characters, the arithmetic hypergeometric $\sD$-modules are overconvergent
isocrystals (not necessarily with Frobenius structure)
if these parameters satisfy the non-Liouville difference condition.
However, since the overholonomicity of arithmetic $\sD$-modules is not preserved by tensor product,
our method is not applicable at least verbatim.
\vspace{10pt}

We briefly explain the strategy of the proof of Main Theorem.

We firstly prove (ii). By a similar argument to the theory of $\sD$-module over the complex numbers,
the convolution with $\sHyp_{\pi}(\alpha_i;\emptyset)$ can be described by using the $p$-adic Fourier transform of $\sD$-modules ``on $\widehat{\bA}^1$''.
We compare $\sHyp_{\pi}(\balpha;\bbeta)$, which is an object ``on $\widehat{\bG_{\rmm}}$'', with
the $\sD$-module ``on $\widehat{\bA}^1$'' associated to $\Hyp_{\pi}(\balpha;\bbeta)$.
(This is a point where we need some $p$-adic calculation.)
By this comparison, we may compute the right-hand side of the isomorphism in (ii)
with the aid of the explicit description of $p$-adic Fourier transform given by Huyghe.

Secondly, (iii) is a corollary of (ii) (see Remark \ref{rem:explicittraces}).

We finally prove (i). It is an essentially classical result that (i) holds for $\sHyp_{\pi}(\alpha_i;\emptyset)$'s and $\sHyp_{\pi}(\emptyset;\beta_j)$'s.
In particular, they are overholonomic $F$-$\sD$-modules.
Because the property of being an overholonomic $F$-$\sD$-module is preserved by six functors as proved by D.\ Caro,
(ii) shows that $\sHyp_{\pi}(\balpha;\bbeta)$ is also an overholonomic $F$-$\sD$-module.
An overholonomic $F$-$\sD$-module is known to be an overconvergent $F$-isocrystal on a dense open subset,
but we have to show that it is so on the ``correct'' open subset.
To do this, we use (iii) to show that $\sHyp_{\pi}(\balpha;\bbeta)$ have the desired rank as an $\sO_{\bX,\bQ}$-module.
By combining this result with some properties of $F$-isocrystals, we can finish the proof of (ii).
\vspace{10pt}

Recently, there appeared an article by Crew \cite{Crew16},
which studies the $p$-adic differential equations from the point of view of rigidity.
He proves the existence of the Frobenius structure of the regular singular rigid overconvergent isocrystals
under suitable conditions.
We note that our Main Theorem (i) in the case where $m=n$ thus proves that
the regular singular differential equations as above are examples of the objects which he is dealing with.

There are other studies on the $p$-adic hypergeometric equations.
Another study by Crew \cite{Crew94CM} deals with them in the case where $n=0$, all $\alpha_i$'s are zero,
$p$ is not divisible by $m$ and $p\neq 2$;
our Main Theorem (i) is proved there in this case.
A theorem due to Tsuzuki \cite[3.3.1, 3.3.3]{Tsuzuki03} is applicable to proving
the overconvergence and the existence of a Frobenius structure for ``Picard--Fuchs equations'' such as $\Hyp(1/2,1/2;0,0)$,
and we may prove Main Theorem (i) by it in these cases.
Our approach is quite different from theirs in the point that
we extensively use the multiplicative convolution and systematically deal with general case including these two cases.

At last, we note that the $p$-adic theory of hypergeometric functions (or exponential sums)
are also studied in terms of partial linear differential equations of rank one on a higher dimensional torus,
for example by Dwork \cite{Dwork} and Adolphson \cite{Adolphson}.
\vspace{10pt}

We conclude this introduction by explaining the structure of the article.

Section 1 is a quick review of the theory of arithmetic $\sD$-modules.

Section 2 is devoted to giving a fundamental properties of multiplicative convolution of arithmetic $\sD$-modules.

Section 3 concerns the arithmetic hypergeometric $\sD$-modules.
In Subsection 3.1, we give a definition of the $p$-adic hypergeometric $\sD$-modules
by using the $p$-adic hypergeometric differential operators, and prove fundamental properties of them.
In Subsection 3.2, we give another definition of the $p$-adic hypergeometric $\sD$-modules
by using the multiplicative convolution, and compare it with the one defined in Subsection 3.1;
Main Theorem (i) is proved here (Theorem \ref{thm:hypandconv}).
Subsection 3.3 gives additional properties of $p$-adic hypergeometric $\sD$-modules
concerning multiplicative convolutions.

Section 4 deals with the hypergeometric isocrystals.
In Subsection 4.1, we proved that the $p$-adic hypergeometric $\sD$-modules are
in fact overconvergent $F$-isocrystals on the desired open subsets.
Here, we finish the proof of Main Theorem ((ii) is proved in Theorem \ref{thm:ocFisoc}. 
For (iii), see Remark \ref{rem:explicittraces}.)
In the last Subsection 4.2, we prove the irreducibility of the hypergeometric isocrystals
and give a characterization of them.

\subsection{Acknowledgements}

The author would like to thank Jeng-Daw Yu for his discussions and a lot of useful comments about this work.
The author would also like to thank Tomoyuki Abe for his answer to my questions about the theory of arithmetic $\sD$-modules.

\subsection{Conventions and Notations.}

Throughout this article, $V$ denotes a complete discrete valuation ring of mixed characteristic $(0,p)$
whose residue field $k$ is a finite field with $q=p^s$ elements.
The fraction field of $V$ is denoted by $K$.
We denote by $|\cdot|$ the norm on $K$ normalized as $|p|=p^{-1}$.
We will assume that $K$ has a primitive $p$-th root of unity after Subsection 3.2.

If $\sP$ is a smooth formal scheme over $\Spf(V)$,
then $\sD^{\dag}_{\sP,\bQ}$ denotes the sheaf of rings of arithmetic differential operators
introduced by Berthelot \cite{BerthelotI}.
Moreover, if $T$ is a divisor of the special fiber of $\sP$, then
$\sD^{\dag}_{\sP,\bQ}(\pdag{T})$ denotes \emph{the sheaf of arithmetic differential operators with
overconvergent singularity along $T$} \cite[4.2]{BerthelotI}.

Let $\sA$ be a sheaf of (not necessarily commutative) rings.
If we say ``$\sA$-module'', we always mean a sheaf of left $\sA$-modules.
$D^{\rb}(\sA)$ denotes the bounded derived category of the category of $\sA$-modules.
$D^{\rb}_{\coh}(\sA)$ denotes the full subcategory of $D^{\rb}(\sA)$
consisting of the complexes whose cohomologies are coherent.

Finally, in this article, if $\kappa$ is a field, a \emph{$\kappa$-variety} means a separated $\kappa$-scheme of finite type.

\section{Arithmetic $\sD$-modules.}

In this section, we recall the theory of arithmetic $\sD$-modules that will be used in this article.

\subsection{Six functors.}
    In this subsection, we fix terminologies and notations,
    and quickly summarize the theory of six-functor formalism of arithmetic $\sD$-modules.
    For a detailed explanation, the reader may consult \cite[Section 1]{AbeCaro}.

    Let $\sP$ be a smooth formal scheme over $\Spf(V)$.
    Caro \cite[D\'efinition 2.1]{Caro09ASENS} defined the subcategory $D^{\rb}_{\ovhol}(\sD^{\dag}_{\sP,\bQ})$ of
    $D^{\rb}_{\coh}(\sD^{\dag}_{\sP,\bQ})$ consisting of \emph{overholonomic $\sD^{\dag}_{\sP,\bQ}$-complexes}.
    Moreover, let $T$ be a divisor of the special fiber of $\sP$.
    The subcategory of $D^{\rb}_{\coh}\big(\sD^{\dag}_{\sP,\bQ}(\pdag{T})\big)$
    consisting of the objects that are overholonomic as $\sD^{\dag}_{\sP,\bQ}$-complex is denoted by
    $D^{\rb}_{\ovhol}\big(\sD^{\dag}_{\sP,\bQ}(\pdag{T})\big)$.
    If a coherent $\sD^{\dag}_{\sP,\bQ}(\pdag{T})$-module is an object
    of $D^{\rb}_{\ovhol}\big(\sD^{\dag}_{\sP,\bQ}(\pdag{T})\big)$, then we say that it is an overholonomic $\sD^{\dag}_{\sP,\bQ}(\pdag{T})$-module.

    It is convenient to introduce here the following terminology.

    \begin{definition}
        (i) A \emph{d-couple} is a pair $(\sP,T)$, where $\sP$ is a smooth formal scheme over $\Spf(V)$ and where $T$ is a divisor of the special fiber of $\sP$ (an empty set is also a divisor).
        
        (ii) A \emph{morphism of d-couples} $\widetilde{f}\colon(\sP',T')\to(\sP,T)$ is
        a morphism $\overline{f}\colon\sP'\to\sP$ such that
        $\overline{f}(\sP'\setminus T')\subset\sP\setminus T$ and that
        $\overline{f}^{-1}(T)$ is a divisor (or empty).
    \end{definition}

    \begin{definition}
        (i) A d-couple $(\sP, T)$ \emph{realizes} the $k$-variety $X$
    if $\sP$ is proper and $X$ is the special fiber of $\sP\setminus T$.

        (ii)
    A morphism of d-couples $\widetilde{f}\colon(\sP',T')\to(\sP,T)$ \emph{realizes} the morphism of $k$-varieties $f\colon X'\to X$
    if $(\sP', T')$ \resp{$(\sP, T)$} realizes $X'$ \resp{$X$} and if $\overline{f}\colon\sP'\to\sP$ induces $f$.
    \end{definition}

    Now, let $X$ be a $k$-variety, and assume that there exists a d-couple $(\sP, T)$ realizing $X$.
    Then, the category $D^{\rb}_{\ovhol}\big(\sD^{\dag}_{\sP,\bQ}(\pdag{T})\big)$ is
    proved to be independent of the choice of $(\sP, T)$ up to a canonical equivalence.
    This category is also denoted by $D^{\rb}_{\ovhol}(X/K)$.

    Even if such a realizing d-couple does not exist,
    the triangulated category $D^{\rb}_{\ovhol}(X/K)$ can be defined for arbitrary realizable $k$-varieties
    (Recall that a $k$-variety is said to be \emph{realizable} if it can be embedded
    into a proper smooth formal scheme over $\Spf(V)$.)
    We do not explain the construction of $D^{\rb}_{\ovhol}(X/K)$ for realizable varieties
    because we have only to use formal properties of this category
    and do not need to refer to the construction.

\begin{definition} Let $X$ be a realizable variety over $k$. (Recall that $k$ is a finite field with $q=p^s$ elements.)
    We denote by $F^{(s)}_{X/k}\colon X\to X$ the $s$-th Frobenius morphism on $X$.
    
    (i) Let $\sM$ be an object of $D^{\rb}_{\ovhol}(X/K)$.
    A \emph{Frobenius structure} on $\sM$ is an isomorphism $\Phi\colon\sM\to (F^{(s)}_{X/k})^{\ast}\sM$.

    (ii) We define the category $F\hyphen D^{\rb}_{\ovhol}(X/K)$ as follows:
    the objects are the pairs $(\sM, \Phi)$ consisting of an object $\sM$ of $D^{\rb}_{\ovhol}(X/K)$ and
    a Frobenius structure $\Phi$ on $\sM$;
    the morphisms are those of $D^{\rb}_{\ovhol}(X/K)$ compatible with the Frobenius structures.
\end{definition}

The category $F\hyphen D^{\rb}_{\ovhol}(X/K)$ is equipped with Grothendieck's six operations.

\vspace{4px}
\textbf{Dual functor.} Let $X$ be a realizable $k$-variety. Then, we have the \emph{dual functor}
    \[
        \bD_X\colon F\hyphen D_{\ovhol}^{\rb}(X/K)\longrightarrow F\hyphen D^{\rb}_{\ovhol}(X/K).
    \]
The functor $\bD_X\circ\bD_X$ is naturally isomorphic to the identity functor on $F\hyphen D^{\rb}_{\ovhol}(X/K)$.

\vspace{4px}
\textbf{Pull-back functors.} Let $f\colon X'\to X$ be a morphism of realizable $k$-varieties.
    Then, we have the extraordinary pull-back functor
    \[
        f^!\colon F\hyphen D^{\rb}_{\ovhol}(X/K)\longrightarrow F\hyphen D^{\rb}_{\ovhol}(X'/K).
    \]
    Moreover, we define the ordinary pull-back functor
    \[
        f^{\plus}\colon F\text{-}D^{\rb}_{\ovhol}(X/K)\longrightarrow F\text{-}D^{\rb}_{\ovhol}(X'/K)
    \]
    by $f^{\plus}=\bD_{X'}\circ f^!\circ\bD_X$.

    If $f$ is an open immersion, then the two functors $f^!$ and $f^{\plus}$ are isomorphic
    \cite[5.5]{Abe14};
    in this case, these two functors are also denoted by $f^{\ast}$.

    Now, let us assume that $f$ is realized by a smooth morphism $\widetilde{f}\colon(\sP',T')\to(\sP,T)$
    of d-couples.
    (A morphism of d-couple is said to be smooth if the morphism $\overline{f}\colon \sP'\to\sP$
    is smooth.)
    In this case, we have a functor
    $\widetilde{f}^!\colon D^{\rb}_{\coh}\big(\sD^{\dag}_{\sP,\bQ}(\pdag{T})\big)\to D^{\rb}_{\coh}\big(\sD^{\dag}_{\sP',\bQ}(\pdag{T'})\big)$
    compatible with the functor $f^!$ above.
    In fact, Berthelot's construction \cite[4.3.3]{Berthelot02Ast} gives
    a functor $\overline{f}^!\colon D^{\rb}_{\coh}\big(\sD^{\dag}_{\sP,\bQ}(\pdag{T})\big)\to
    D^{\rb}_{\coh}\big(\sD^{\dag}_{\sP',\bQ}(\pdag{f^{-1}(T)})\big)$,
    and we have $\widetilde{f}^!(\sM)=\sD^{\dag}_{\sP',\bQ}(\pdag{T'})\otimes\overline{f}^!(\sM)$.

    In the case where $f\colon \sX'\hookrightarrow\sX$ is an open immersion, where $\sP'=\sP$ and where $\overline{f}$ is the identity morphism,
    we also denote $\widetilde{f}^!$ by $\widetilde{f}^{\ast}$.
    The functor $\widetilde{f}^{\ast}$ is exact on the category of coherent $\sD^{\dag}_{\sP,\bQ}(\pdag{T})$-modules
    because a sequence of coherent $\sD^{\dag}_{\sP,\bQ}(\pdag{T'})$-modules is exact
    if and only if its restriction on $\sX'$ is \cite[Proposition 4.3.12 (ii)]{BerthelotI}.

\vspace{4px}
\textbf{Push-forward functors.}
    Let $f\colon X'\to X$ be a morphism of realizable $k$-varieties.
    Then, we have the push-forward functor $f_{\plus}$:
    \[
        f_{\plus}\colon F\hyphen D^{\rb}_{\ovhol}(X'/K)\longrightarrow F\hyphen D^{\rb}_{\ovhol}(X/K).
    \]

    Moreover, we define the extraordinary push-forward functor
    \[
        f_!\colon F\hyphen D^{\rb}_{\ovhol}(X'/K)\longrightarrow F\hyphen D^{\rb}_{\ovhol}(X/K)
    \]
    by $f_!=\bD_X\circ f_{\plus}\circ\bD_{X'}$.

    We have a natural morphism $f_!\to f_{\plus}$, and it is an isomorphism if $f$ is proper.

    We mention here a special case.
    Assume that $f$ is realized by a morphism of d-couples $\widetilde{f}\colon(\sP',T')\to(\sP,T)$.
    Then, we have a push-forward functor \cite[1.1.6]{Caro06CM}
    \[
        \widetilde{f}_{\plus}\colon D^{\rb}_{\coh}\big(\sD^{\dag}_{\sP',\bQ}(\pdag{T'})\big)\to D^{\rb}\big(\sD^{\dag}_{\sP,\bQ}(\pdag{T})\big)
    \]
    and it is compatible with the $f_{\plus}$ above.
    If $f$ is an open immersion, if $\sP'=\sP$ and if $\overline{f}$ is the identity morphism on $\sP$,
    then $\widetilde{f}_{\plus}$ is obtained by considering the complex of $\sD^{\dag}_{\sP,\bQ}$-modules
    as a complex of $\sD^{\dag}_{\sP,\bQ}(\pdag{T})$-module
    via the inclusion $\sD^{\dag}_{\sP,\bQ}(\pdag{T})\hookrightarrow\sD^{\dag}_{\sP,\bQ}(\pdag{T'})$.

\vspace{4px}
\textbf{Tensor products.}
    Let $X$ be a realizable $k$-variety. Then we have the twisted tensor functor
    \[
        \widetilde{\otimes}\colon F\hyphen D^{\rb}_{\ovhol}(X/K)\times F\hyphen D^{\rb}_{\ovhol}(X/K)\longrightarrow
        F\hyphen D^{\rb}_{\ovhol}(X/K).
    \]
    If $f\colon X'\to X$ is a morphism of $k$-varieties, then we have an isomorphism
    $f^!( \text{ -- } \totimes \text{ -- } )\cong f^!(\text{ -- })\totimes f^!(\text{ -- })$.

    If $X$ is realized by a d-couple $(\sP, T)$, 
    then for two coherent $\sD^{\dag}_{\sP,\bQ}(\pdag{T})$-modules $\sM$ and $\sN$,
    the module $\sM\widetilde{\otimes}\sN$ is defined by $\sM\otimes^{\dag}_{\sO_{\sP}(\pdag{T})_{\bQ}}\sN[-\dim\sP]$.
    Here, $\sO_{\sP,\bQ}(\pdag{T})$ denotes $\sO_{\sP}(\pdag{T})\otimes\bQ$,
    where $\sO_{\sP}(\pdag{T})$ is the sheaf of functions on $\sP$ with
    overconvergent singularities along $T$ \cite[4.2.4]{BerthelotI},
    and $\otimes^{\dag}$ denotes the overconvergent tensor product.

    The tensor functor $-\otimes-\defeq \bD_X\big(\bD_X(-)\totimes\bD_X(-)\big)$ is also defined.

\vspace{4px}
\textbf{Exterior tensor product.}
    Let $X, Y$ be two realizable $k$-varieties.
    Then, we have the \emph{exterior tensor product}
    \[
        \boxtimes\colon F\hyphen D_{\ovhol}^{\rb}(X/K)\times F\hyphen D_{\ovhol}^{\rb}(Y/K) \longrightarrow
        F\hyphen D^{\rb}_{\ovhol}(X\times Y/K)
    \]
    defined by
    $\sM\boxtimes\sN\defeq \pr_1^{!}\sM\widetilde{\otimes}\pr_2^{!}\sN$,
    where $\pr_1\colon X\times Y\to X$ \resp{$\pr_2\colon X\times Y\to Y$} denotes the 
    first \resp{second} projection.

    The functor $-\boxtimes-$ is isomorphic to $\pr_1^{\plus}(-)\otimes\pr_2^{\plus}(-)$
    \cite[Proposition 1.3.3]{AbeCaro}.

    \label{para:boxtimesandpush}
    We also have the relative K\"unneth formula \cite[1.1.5]{Abe13Langlands}.
    Namely, if $f\colon X'\to X$, $g\colon Y'\to Y$ are morphisms of realizable $k$-varieties,
    then there exists a natural isomorphism $(f\times g)_{\plus}(\sM\boxtimes\sN)\cong(f_{\plus}\sM)\boxtimes (g_{\plus}\sN)$
    for each object $\sM$ of $F\hyphen D^{\rb}_{\ovhol}(X'/K)$ and $\sN$ of $F\hyphen D^{\rb}_{\ovhol}(Y'/K)$.
    By using the dual functor, we also get a natural isomorphism $(f\times g)_!(\sM\boxtimes\sN)\cong(f_!\sM)\boxtimes(g_!\sN)$.

\vspace{4px}
Projection formula is also available \cite[A.6]{AbeCaro}.
Namely, let $f\colon X'\to X$ be a morphism of realizable $k$-varieties.
Then, for each $\sM\in F\hyphen D^{\rb}_{\ovhol}(X'/K)$ and $\sN\in F\hyphen D^{\rb}_{\ovhol}(X/K)$, we have
natural isomorphisms $f_{\plus}(\sM\widetilde{\otimes}f^{!}\sN)\cong f_{\plus}(\sM)\widetilde{\otimes}\sN$.

\vspace{4px}
The base change theorem is also available;
assume that we are given the cartesian diagram
    \[
    \begin{tikzpicture}[description/.style={fill=white,inner sep=2pt}]
            \matrix (m) [matrix of math nodes, row sep=0.4em, column sep=1em, text height=1.5ex, text depth=0.25ex]
            {
                X_2' & & X_2 \\
                & \square & \\
                X_1' & & X_1 \\
            };
            \path[->,font=\scriptsize]
                (m-1-1) edge node[above] {$g'$} (m-1-3)
                (m-1-1) edge node[left] {$f'$} (m-3-1)
                (m-1-3) edge node[right] {$f$} (m-3-3)
                (m-3-1) edge node[below] {$g$} (m-3-3);
    \end{tikzpicture}
    \]
of realizable $k$-varieties. 
Then, we have natural isomorphisms $g^!\circ f_{\plus}\cong f'_{\plus}\circ g'^!$ and
$g^{\plus}\circ f_!\cong f'_!\circ g'^{\plus}$ of functors
$F\hyphen D^{\rb}_{\ovhol}(X_2/K)\to F\hyphen D^{\rb}_{\ovhol}(X_1'/K)$.
In fact, the first isomorphism is the base change theorem \cite[1.3.10]{AbeCaro},
    and the second one is derived from the first by using the dual functors.

    \vspace{4px}
\textbf{Scalar Extension.}
We introduce a functor of scalar extension.
Let $k'$ be a finite extension of $k$ with $q'=p^{s'}$ elements,
and put $V'\defeq V\otimes_{W(k)}W(k')$ and $K'\defeq\Frac(V')$.
Let $X$ be a realizable variety over $k$ and put $X'\defeq X\times_kk'$.
Under this setting, we have the ``changing the base field'' functor
\[
    \iota_{k'/k}\colon F\hyphen D^{\rb}_{\ovhol}(X/K)\to F\hyphen D^{\rb}_{\ovhol}(X'/K').
\]
(Note that a Frobenius structure on an object $\sN$ of $D^{\rb}_{\ovhol}(X'/K')$ is defined to be
an isomorphism $\sN\to(F^{(s')}_{X'})^{\ast}\sN$.)

For later use, we describe the construction of $\iota_{k'/k}$ 
assuming that $X$ is realized by a d-couple $(\sP,T)$ realizing $X$.
We first fix notations.
Put $\sP'\defeq\sP\times_{\Spf(V)}\Spf(V')$ and $T'\defeq T\times_kk'$,
and denote by $f\colon\sP'\to\sP$ the projection.
In this paragraph, we also denote $\sD^{\dag}_{\sP,\bQ}$ by $\sD^{\dag}_{\sP/V,\bQ}$
to make it explicit that it is defined for the morphism $\sP\to\Spf(V)$.

Now, take an object $(\sM,\Phi)$ of $F\hyphen D^{\rb}_{\ovhol}(X/K)$.
Let us put
\[
    \sM'\defeq\sD^{\dag}_{\sP'/V',\bQ}\otimes_{f^{-1}\sD^{\dag}_{\sP/V,\bQ}} f^{-1}\sM.
\]
Since $f^{\ast}\sD^{\dag}_{\sP/V,\bQ}\cong\sD^{\dag}_{\sP'/V',\bQ}$ \cite[2.2.2]{Berthelot02Ast},
$\sM'$ is isomorphic to $f^!\sM$ as objects of $D^{\rb}_{\ovhol}(X'/K)$
by definition of Berthelot's functor $f^!$.
Because of the presence of the Frobenius structure, $\sM'$ is in fact an object of $D^{\rb}_{\ovhol}(X'/K')$
\cite[1.2]{AbeCaro14}.
Moreover, by using the (semi-linear) morphism $\Phi'\colon \sM'\to (F_{X'}^{(s)})^{\ast}\sM'$ induced by $\Phi$,
we define an isomorphism $\Phi''\colon \sM'\to (F_{X'}^{(s')})^{\ast}\sM'$ in $D^{\rb}_{\ovhol}(X'/K')$ by
\[
    \Phi''\defeq (F_{X'}^{(s'-s)})^{\ast}\Phi'\circ(F_{X'}^{(s'-2s)})^{\ast}\Phi'\dots\circ(F_{X'}^{(s)})^{\ast}\Phi'\circ\Phi'.
\]
We define $\iota_{k'/k}\big((\sM,\Phi)\big)$ to be the pair $(\sM',\Phi'')$.
By the properties of six functors, the scalar extension functor $\iota_{k'/k}$ commutes with six functors.

\vspace{4pt}
\textbf{Frobenius trace function.}
Let $X$ be a realizable $k$-variety, let $(\sM, \Phi)$ be an object of $F\hyphen D^{\rb}_{\ovhol}(X/K)$
and let $k'$ be a finite extension of $k$.
In this paragraph, we make explicit the term ``Frobenius trace of $(\sM,\Phi)$ at a $k'$-valued point $x\in X(k')$''.
Let $i_x\colon\Spec(k')\hookrightarrow X'$ denote the morphism of $k'$-varieties induced by $x$.
Under the notation in the previous paragraph, put
$(\sM',\Phi')\defeq i_x^{\plus}\circ \rho_{k'/k}\big( (\sM,\Phi)\big)$;
it is thus an object of $F\hyphen D^{\rb}_{\ovhol}(\Spec(k')/K')$.
$\sM'$ is a complex of $K'$-vector spaces each of whose cohomology is finite-dimensional,
and the natural morphism $\sM'\to(F^{(s')})^{\ast}\sM'$ composed with $(\Phi')^{-1}\colon(F^{(s')})^{\ast}\sM'\to\sM'$ is a $K'$-linear isomorphism of $\sM'$,
which we denote by $F_{\sM,x}$.
Now, the Frobenius trace of $(\sM,\Phi)$ at $x\in X(k')$ is defined to be
$\sum_{i\in\bZ}(-1)^i\tr\big(F_{\sM,x}|\sH^i(\sM')\big)$.

\vspace{4pt}
\textbf{Overconvergent $F$-isocrystals.}
At last, we recall the relationship between the theory of isocrystals and that of arithmetic $\sD$-modules.
Let $X$ be a smooth scheme of dimension $d$ which is separated of finite type over $k$,
and assume that there exists a d-couple $(\sP, T)$ realizing $X$.
Then, there exists a fully faithful functor
\[
    \sp_{\plus}\colon F\hyphen\Isoc^{\dag}(X/K)\longrightarrow F\hyphen\Coh\big(\sD^{\dag}_{\sP,\bQ}(\pdag{T})\big),
\]
where the target denotes the category of coherent $\sD_{\sP,\bQ}(\pdag{T})$-modules with Frobenius structure.
The essential image of this functor is the coherent $\sD_{\sP,\bQ}(\pdag{T})$-modules
whose restriction to $\sP\setminus T$ is coherent as an $\sO_{\sP\setminus T,\bQ}$-module \cite[Th\'eor\`eme 2.2.12]{Caro06CM}.
We say that an object $F\hyphen\Coh\big(\sD_{\sP,\bQ}^{\dag}(\pdag{T})\big)$ is an overconvergent $F$-isocrystal
if it belongs to the essential image of $\sp_{\plus}$.

If $f\colon X\to Y$ is a morphism between smooth schemes separated of finite type,
and if we put $d\defeq \dim X-\dim Y$,
we have $\sp_{\plus}\circ f^{\ast}\cong f^{\plus}([-d])\circ\sp_{\plus}$ \cite[4.1.1]{Abe13Langlands}.
Moreover, if we $\sP$ is of constant dimension, then 
we have $\sp_{\plus}(-\otimes -)\cong \sp_{\plus}(-)\totimes\sp_{\plus}(-)[\dim\sP]$ \cite[3.1.8]{Caro15MM}.

\subsection{$\sD^{\dag}$-affinity.}

As in the classical case, when we discuss about the coherent $\sD^{\dag}$-modules,
it often suffices that we may just discuss on the global sections.
In this subsection, we recall a fundamental theorem about this point
and give two particular examples which we regularly use in this article.

\begin{theorem}[{\cite[5.3.3]{Huyghe98AIF}}]
    Let $\sP$ be a projective smooth formal scheme over $\Spf(V)$, and let $T$ be an ample divisor of
    the special fiber of $\sP$.
    Then, the functor $\Gamma(\sP, -)$ is exact and gives an equivalence of categories
    from the category of coherent $\sD^{\dag}_{\sP,\bQ}(\pdag{T})$-modules
    to the category of coherent $\Gamma\big(\sP, \sD^{\dag}_{\sP,\bQ}(\pdag{T})\big)$-modules.
    \label{thm:Daffine}
\end{theorem}

We particularly apply this theorem to the following two d-couples:
$(\sP, T)=(\widehat{\bP^1_V}, \{\infty\})$ and $(\sP, T)=(\widehat{\bP^1_V}, \{0,\infty\})$.
In these cases, the ring of global sections of the sheaf $\sD^{\dag}_{\sP,\bQ}(\pdag{T})$ has an explicit description (cf. \cite[p. 915]{Huyghe98AIF}),
and we immediately get the following corollaries.

\begin{corollary}
\label{para:DoverA1}
    Let $A_1(K)^{\dag}$ be the ring defined by
    \[ 
        A_1(K)^{\dag} \defeq \Set{\sum_{l,k\in\bN}a_{l,k}x^l\partial^{[k]} | a_{l,k}\in K, \exists C>0, \exists \eta<1, |a_{l,k}|_p < C\eta^{l+k}}.
    \]
    Then, the functor $\Gamma(\widehat{\bP^1_V}, -)$ is exact and gives an equivalence of categories
    from the category of coherent $\sD_{\widehat{\bP^1_V},\bQ}^{\dag}(\pdag\{\infty\})$-modules to
    that of coherent $A_1(K)^{\dag}$-modules.
\end{corollary}

\begin{corollary}
\label{para:DoverGm}
    Let $B_1(K)^{\dag}$ be the ring defined by
    \[
        B_1(K)^{\dag} \defeq \Set{\sum_{l\in\bZ,k\in\bN}a_{l,k}x^l\partial^{[k]} | a_{l,k}\in K, \exists C>0, \exists \eta<1,
        |a_{|l|_{\infty},k}|_p < C\eta^{|l|_{\infty}+k}}.
    \]
    Here, $|l|_{\infty}$ denotes the Euclid norm.
    Then, the functor $\Gamma(\widehat{\bP^1_V}, -)$ is exact and gives an equivalence of categories
    from the category of coherent $\sD^{\dag}_{\widehat{\bP^1_V},\bQ}(\pdag\{0,\infty\})$-modules
    to that of coherent $B_1(K)^{\dag}$-modules.
\end{corollary}

In the rest of this article, we implicitly use these equivalences and identify the source and the target.

\subsection{Examples of overholonomic arithmetic $\sD$-modules.}

Now, we introduce some specific examples of overholonomic $F$-$\sD^{\dag}$-modules.

\vspace{4px}
\textbf{Delta modules.}
    Let $X$ be an open subscheme of $\bA^1_k$,
    let $\lambda$ be an element of $k$ and let
    $\iota_{\lambda}\colon \{\lambda\}\hookrightarrow X$ denote the inclusion.
    We define an object $\delta_{\lambda}$ of $F\hyphen D^{\rb}_{\ovhol}(X/K)$, the ``Dirac delta module'' at $\lambda$.
    Note that the sheaf $\sO_{\{\lambda\},\bQ}\defeq\sO_{\{\lambda\}}\otimes\bQ$ can be seen as an object of $F\hyphen D^{\rb}_{\ovhol}(\{\lambda\}/K)$.
    We define $\delta_{\lambda}\defeq \iota_{\lambda,\plus}(\sO_{\{\lambda\},\bQ})$.
    If $\widetilde{\lambda}$ denotes the Teichm\"uller lift of $\lambda$, 
    this object is isomorphic to $A_1(K)^{\dag}/A_1(K)^{\dag}(x-\widetilde{\lambda})$
    as a coherent $A_1(K)^{\dag}$-module.

\vspace{4px}
\textbf{Dwork module associated with a non-trivial additive character.}
    Let $\psi$ be a non-trivial additive character on $k$
    and assume that $K$ contains a primitive $p$-th root of unity.
	We associate with $\psi$ an overholonomic $F\hyphen\sD^{\dag}_{\widehat{\bP^1_V},\bQ}(\pdag{\{\infty\}})$-module (therefore an object of $F\hyphen D^{\rb}_{\ovhol}(\bA^1_k)$) $\sL_{\psi}$,
	which we call ``the Dwork module associated with $\psi$''.
    Before we review its definition, let us recall the fact that $\psi$ defines an element $\pi_{\psi}$ of $K$ as follows \cite[(1.3)]{Berthelot84Ast}.

    Firstly, remark that the character $\psi$ takes value in $K$ because we assumed that
    $K$ has a primitive $p$-th root of unity.
    Now, if $k$ is a prime field, then $\pi_{\psi}$ is defined to be the unique root of $X^{p-1}+p$ that satisfies
    \[
        \psi(1)\equiv 1+\pi_{\psi} \pmod{\pi_{\psi}^2}.
    \]
    For general $k$, let $k_0$ be the prime field contained in $k$ and fix a nontrivial additive character $\psi_0$ on $\bF_p$.
    Then, there exists a unique element $a\in\bF_q$ that satisfies
    \[
        \psi(x) = \psi_0\big(\Tr_{k/k_0}(ax)\big)\quad \forall x\in k.
    \]
	The element $\pi_{\psi}$ is defined by $\pi_{\psi}=\pi_{\psi_0}\widetilde{a}$, where $\widetilde{a}$ is
    the Teichm\"uller lift of $a$ in $K$.

    Conversely, we may recover $\psi$ from $\pi_{\psi}$ as follows.
    In fact, the radius of convergence of the formal power series $\theta_{\psi}(z)=\exp\big(\pi_{\psi}(z-z^q)\big)$ is strictly
    greater that $1$, and for each $x\in\bF_q$, we have the equation $\theta_{\psi}(\widetilde{x})=\psi(x)$
    \cite[Lemme (1.4)]{Berthelot84Ast}.
    We note that an element $\pi$ of $K$ is of the form $\pi_{\psi}$ for some non-trivial additive character $\psi$ on $k$ if and only if $\pi^{q-1}=(-p)^{(q-1)/(p-1)}$. In fact, $\pi_{\psi}^{q-1}=(-p)^{(q-1)/(p-1)}$
    by the definition of $\pi_{\psi}$, and since different two non-trivial characters on $k$ give
    different $\pi_{\psi}$'s, each root of $X^{q-1}-(-p)^{(q-1)/(p-1)}$ must be of the form $\pi_{\psi}$.

    We are now ready for defining the Dwork module $\sL_{\psi}$ associated with the non-trivial additive character $\psi$.
    It is defined to be the coherent $\sD_{\widehat{\bP^1_V},\bQ}^{\dag}(\pdag{\{\infty\}})$-module
    \[
        \sL_{\psi}\defeq \sD_{\widehat{\bP^1_V},\bQ}^{\dag}(\pdag{\{\infty\}})/\sD_{\widehat{\bP^1_V},\bQ}^{\dag}(\pdag{\{\infty\}})(\partial+\pi_{\psi})
    \]
    with the Frobenius structure
    \begin{equation}
        \Phi\colon\sL_{\psi}\to (F^{(s)}_{\bP^1_k/k})^{\ast}\sL_{\psi};\quad e_{\psi}\mapsto \exp\big(-\pi_{\psi}(x-x^q)\big)\otimes e_{\psi};
        \label{eq:FrobeniusofDwork}
    \end{equation}
    here, $e_{\psi}$ denotes the global section of $\sL_{\psi}$ defined by the global section $1\in A_1(K)^{\dag}$
    and $x$ denotes the coordinate.
    We also denote $\sL_{\psi}$ by $\sL_{\pi_{\psi}}$ and call it the Dwork module associated with $\pi_{\psi}$.

    We recall some fundamental properties of the Dwork module.

    \begin{proposition}
        Let $\psi$ be a non-trivial additive character on $k$.
        
        \textup{(i)} Let $L_{\psi}$ be the Dwork $F$-isocrystal associated to $\psi$ \cite{Berthelot84Ast}.
        Then, ${\sp}_{\plus}L_{\psi}$ is isomorphic to $\sL_{\psi}$.

        \textup{(ii)} Let $k'$ be a finite extension of $k$.
        Then, the base extension $\iota_{k'/k}(\sL_{\psi})$ of $\sL_{\psi}$ is isomorphic to $\sL_{\psi\circ\Tr_{k'/k}}$.

        \textup{(iii)} Let $x$ be an element of $k$ and consider it as a $k$-valued point $i_{x}\colon \Spec(k)\hookrightarrow\bA_k^1$.
        Then, Frobenius trace of $\sL_{\psi}$ at $x$ is $-q\psi(x)$.
        \label{prop:propertiesofDwork}
    \end{proposition}
    \begin{proof}
        (i) It is a classical fact \cite[5.2]{Berthelot90}.

            (ii) Because $\iota_{k'/k}$ is exact and sends $\sD^{\dag}_{\widehat{\bP^1_V},\bQ}(\pdag{\{\infty\}})$ to $\sD^{\dag}_{\widehat{\bP^1_{V'}},\bQ}(\pdag{\{\infty\}})$,
            we have
            \[
                \iota_{k'/k}(\sL_{\psi})=\sD^{\dag}_{\widehat{\bP^1_{V'}}}(\pdag{\{\infty\}})/\sD^{\dag}_{\widehat{\bP^1_{V'}}}(\pdag{\{\infty\}})(\partial+\pi_{\psi})
            \]
            as a $\sD^{\dag}_{\widehat{\bP^1_{V'}}}(\pdag{\{\infty\}})$-module.
            Since $\pi_{\psi}=\pi_{\psi\circ\Tr_{k'/k}}$, it is the same object as $\sL_{\psi\circ\Tr_{k'/k}}$.
            The Frobenius structures are also compatible because $\theta_{\psi\circ\Tr_{k'/k}}(z)=\prod_{i=1}^{s'-s}\theta_{\psi}(z^{q^{i-1}})$.

        (iii) 
        Because of the isomorphism $i_x^{\plus}\sp_{\plus}(L_{\psi})\cong \sp_{\plus}(i_x^{\ast}L_{\psi})(1)[1]$,
        it suffices to prove that the Frobenius trace of $L_{\psi}$ at $x$ equals $\psi(x)$.
        Since $L_{\psi}$ is an isocrystal of rank one,
		$i_x^{\ast}L_{\psi}$ is a one-dimensional $K$-vector space.
        If $x\neq 0$, let $\widetilde{x}$ denote the Teichm\"uller lift of $x$;
        then the Frobenius structure is defined by
        \[
            i_x^{\ast}L_{\psi}\to i_x^{\ast}L_{\psi};
        \quad e\mapsto \exp\big(\pi_{\psi}(\widetilde{x}-\widetilde{x}^q)\big)e=\theta_{\psi}(\widetilde{x})e=\psi(x)e,
        \]
        which shows the claim.
        A similar calculation shows the claim in the case where $x=0$.
    \end{proof}

\vspace{4px}
\textbf{Kummer module associated with a multiplicative character.}
    Let $\alpha$ be an element of $\frac{1}{q-1}\bZ$.
	We associate with $\alpha$ an overholonomic $F\hyphen\sD^{\dag}_{\widehat{\bP^1}_V,\bQ}(\pdag{\{0,\infty\}})$-module (therefore an object of $F\hyphen D^{\rb}_{\ovhol}(\bG_{\rmm,k})$) $\sK_{\alpha}$, which we call
	``the Kummer module associated with $\alpha$'', in the following way.
	We define $\sK_{\alpha}$ to be $B_1(K)^{\dag}/B_1(K)^{\dag}(x\partial-\alpha)$ (considered as
    a coherent $\sD^{\dag}_{\bP^1,\bQ}(\pdag\{0,\infty\})$-module)
    with the Frobenius structure
    \begin{equation}
        \sK_{\alpha}\to (F_{\bP^1_k/k}^{(s)})^{\ast}\sK_{\alpha}; \quad e_{\alpha}\mapsto x^{-\alpha(q-1)}\otimes e_{\alpha};
        \label{eq:FrobeniusofKummer}
    \end{equation}
    here, $e_{\alpha}$ denotes the global section of $\sK_{\alpha}$ defined by the element $1\in B_1(K)^{\dag}$.

	As we did in the case of Dwork modules, we may consider the Kummer modules as being associated with characters.
	Let $\chi$ be a (multiplicative) character on $k^{\times}$.
    Let $\alpha_{\chi}$ be an element of $\frac{1}{q-1}\bZ$ satisfying $\chi(\xi)=\widetilde{\xi}^{(q-1)\alpha_{\chi}}$. 
    (This condition does not uniquely determine $\alpha_{\chi}\in\frac{1}{q-1}\bZ$, but
    $\alpha_{\chi} \bmod \bZ$ is a uniquely determined element of $\frac{1}{q-1}\bZ/\bZ$.)
    $\sK_{\alpha_{\chi}}$ is also denoted by $\sK_{\chi}$; this does not depend on the choice of $\alpha_{\chi}$
    up to isomorphism because $\sK_{\alpha}\cong\sK_{\alpha+1}$ for any $\alpha\in\frac{1}{q-1}\bZ$.

    The following properties, which correspond to those for Dwork modules (Proposition \ref{prop:propertiesofDwork}), are also available.
    The proof is also parallel to that of Proposition \ref{prop:propertiesofDwork} and we omit it.
        
    \begin{proposition}
        Let $\chi$ be a character on $k^{\times}$.
        
        \textup{(i)} Let $K_{\alpha}$ be the Kummer convergent $F$-isocrystal on $\bG_{\rmm,k}$ overconvergent along $\{0,\infty\}$ defined by $\alpha$.
        Then, $\sp_{\plus}K_{\alpha}$ is isomorphic to $\sK_{\chi}$.

        \textup{(ii)} Let $k'$ a finite extension of $k$.
        Then, the base extension $\iota_{k'/k}(\sK_{\psi})$ of $\sK_{\psi}$ is isomorphic to $\sK_{\psi\circ\Norm_{k'/k}}$.

        \textup{(iii)} Let $x$ be an element of $k'^{\times}$ and consider it as a $k$-valued point $i_x\colon\{x\}\hookrightarrow\bG_{\rmm,k}$.
        Then, the Frobenius trace of $\sK_{\chi}$ at $x$ is $-q\chi(x)$.
        \label{prop:propertiesofKummer}
    \end{proposition}

\subsection{Fourier transform.}

Let $\psi$ be a non-trivial additive character on $k$,
and we fix it throughout this subsection.
We assume that $K$ contains a primitive $p$-th root of unity.
Recall from the previous subsection that then an element $\pi_{\psi}$ of $K$ is associated with $\psi$
and that we have the Dwork module $\sL_{\psi}$ in $F\hyphen D^{\rb}_{\ovhol}(\bA^1_k/K)$.

The purpose of this subsection is recalling the theory of Fourier transforms of arithmetic $\sD$-modules, which is closely related to the theory of multiplicative convolution.
The basic references are articles of Huyghe \cite{NootHuyghe04, NootHuyghe13}.
Her theory of $p$-adic Fourier transforms is the one with respect to the character $\psi$
which can be expressed by $\psi_0\circ\Tr_{k/k_0}$ (in other words, for which $\pi_{\psi}$ is a root of $X^{p-1}+p$).
However, her argument remains valid for general $\psi$,
we use the theory of $p$-adic Fourier transform for a general $\psi$.

First, we define the ``integral kernel'' of the arithmetic Fourier transform.
Let $\mu\colon \bA^1_k\times\bA^1_k\to\bA^1_k$ be the multiplication morphism $(x,y)\mapsto xy$,
and define $\sL_{\mu,\psi}\defeq \mu^!(\sL_{\psi}[-1])$.
Then, $\sL_{\mu,\psi}$ is an object of $F\hyphen D^{\rb}_{\ovhol}(\bA^1_k\times\bA^1_k/K)$
concentrated on degree $0$.
Let $(\sP', T')$ denote the d-couple defined by $\sP'=\widehat{\bP^1_V}\times\widehat{\bP^1_V}$ and
$T'=(\{\infty\}\times\bP^1_k)\cup(\bP^1_k\times\{\infty\})$.
We also regard $\sL_{\mu,\psi}$ as an object of $D^{\rb}_{\coh}\big(\sD^{\dag}_{\sP',\bQ}(\pdag{T'})\big)$.

\begin{definition}
    The functor
    \[
        \FT_{\psi}\colon D^{\rb}_{\coh}\big(\sD^{\dag}_{\widehat{\bP^1_V},\bQ}(\pdag{\{\infty\}})\big)
        \longrightarrow D^{\rb}_{\coh}\big(\sD^{\dag}_{\widehat{\bP^1_V},\bQ}(\pdag{\{\infty\}})\big)
    \]
    is defined by sending $\sM$ in $D^{\rb}_{\coh}\big(\sD^{\dag}_{\widehat{\bP^1_V},\bQ}(\pdag{\{\infty\}})\big)$ to
    \[
        \FT_{\psi}(\sM) = \widetilde{\pr_2}_{,\plus}\big(\sL_{\mu,\psi}\otimes^{\dag}_{\sO_{\sP',\bQ}(\pdag{T'})}\widetilde{\pr_1}^!\sM\big)[-2],
    \]
    where $\widetilde{\pr_i}\colon (\sP',T')\to(\widehat{\bP^1_V},\{\infty\})$ denotes the smooth morphism of d-couples
    defined by the $i$-th projection $\widehat{\bP^1_V}\times\widehat{\bP^1_V}\to\widehat{\bP^1_V}$ for $i\in\{1,2\}$.
    This object $\FT_\psi(\sM)$ is called the \emph{geometric Fourier transform} of $\sM$.

    We also define the geometric Fourier transform on the category
    $F\hyphen D^{\rb}_{\ovhol}(\bA^1_k/K)$. Namely, we define
    \[
        \FT_{\psi}\colon F\hyphen D^{\rb}_{\ovhol}(\bA^1_k/K)\to F\hyphen D^{\rb}_{\ovhol}(\bA^1_k/K);\quad
        \sM\mapsto\pr_{2,+}\big(\sL_{\mu,\psi}\totimes\pr_1^!\sM\big),
    \]
    where $\pr_i$'s are $i$-th projections $\bA_k^1\times\bA_k^1\to\bA^1_k$.
\end{definition}

These two functors are compatible with each other because it is true for each functor appearing in the definitions. 

In defining the geometric Fourier transform,
we could have used $\pr_{2,!}$ instead of $\pr_{2,\plus}$,
and Huyghe proved that these two ``geometric Fourier transforms'' coincide.
The following proposition is a special case of her result.

\begin{proposition}[{\cite[Theorem 3.2]{NootHuyghe13}}]
    Let $\sM$ be an object of $F\hyphen D_{\ovhol}^{\rb}(\bA^1/K)$.
    Then, the natural morphism
    \[
        \pr_{2,!}\big(\sL_{\mu,\psi}\totimes\pr_1^!\sM\big) \longrightarrow \FT_{\psi}(\sM)
    \]
    is an isomorphism.
    \label{prop:miracleofft}
\end{proposition}

\begin{remark}
    The original form \cite[Theorem 3.2]{NootHuyghe13} of this proposition by Huyghe is quite stronger.
    It states that for all objects of $D_{\coh}^{\rb}\big(\sD^{\dag}_{\widehat{\bP^1_V},\bQ}(\pdag\{\infty\})\big)$,
    the source of this morphism is well-defined, that so is the morphism itself, and that it is an isomorphism.
\end{remark}

Under the identification in \ref{para:DoverA1}, the geometric Fourier transform of an overholonomic $\sD^{\dag}_{\widehat{\bP^1_V},\bQ}(\pdag\{\infty\})$-module
is explicitly described as follows.

\begin{proposition}[{\cite[Th\'eor\`eme 5.3.1]{NootHuyghe04}}]
    \label{prop:geomftandnaiveft}
    Let $\varphi_{\pi_{\psi}}\colon A_1(K)^{\dag}\to A_1(K)^{\dag}$ be the ring automorphism defined by $\varphi_{\pi_{\psi}}(x)=-\partial/\pi_{\psi}$ and $\varphi_{\pi_{\psi}}(\partial)=\pi_{\psi} x$.
    Let $\sM$ be a coherent $A_1(K)^{\dag}$-module
    and denote by $\varphi_{\pi_{\psi},\ast}\sM$ the coherent $A_1(K)^{\dag}$-module obtained by
    letting $A_1(K)^{\dag}$ act on $\sM$ via $\varphi_{\pi_{\psi}}$.
    Then, we have a natural isomorphism $\FT_{\psi}(\sM) \cong \varphi_{\pi_{\psi},\ast}\sM[-1].$
    \label{prop:globalsectionofft}
\end{proposition}

\section{Multiplicative Convolution of Arithmetic $\sD$-modules.}

In this section, we construct a theory of multiplicative convolutions of arithmetic $\sD$-modules.
The contents of this section is basically a direct translation of the classical arguments
for the convolution of complex $\sD$-modules \cite[Chapter 5]{Katz90ESDE} and for
$\ell$-adic perverse sheaves \cite[Chapter 7]{Katz90ESDE} to the $p$-adic setting.

\subsection{Definition and basic properties.}

\begin{definition}
    Let $\sM, \sN$ be objects of $F$-$D^{\rb}_{\ovhol}(\bG_{\rmm,k}/K)$.
    Then, we define two \emph{convolutions}, $\sM\ast_{!}\sN$ and $\sM\ast_{\plus}\sN$, by
    \[
        \sM\ast_{!}\sN\defeq\mu_{!}\big(\sM\boxtimes\sN\big),\quad \sM\ast_{\plus}\sN\defeq \mu_{\plus}\big(\sM\boxtimes\sN\big),
    \]
    where $\mu\colon\bG_{\rmm,k}\times\bG_{\rmm,k}\to\bG_{\rmm,k}$ denotes the multiplication morphism.
\end{definition}

We here state some basic properties of these convolutions.

\begin{proposition}
    Let $\ast$ be the convolution $\ast_{!}$ or $\ast_{\plus}$.
    Then, the following properties hold.
    \begin{enumerate}
        \item[\textup{(i)}] The delta module $\delta_1$ is the unit of the convolution,
            that is, we have natural isomorphisms
            $\delta_1\ast \sM\cong\sM$ and $\sM\ast\delta_1\cong\sM$ for all
            $\sM\in F\hyphen D^{\rb}_{\ovhol}(\bG_{\rmm,k}/K)$.
        \item[\textup{(ii)}] The convolution is commutative, that is, 
            we have a natural isomorphism $\sM_1\ast\sM_2\cong\sM_2\ast\sM_1$
            for all $\sM_1,\sM_2\in F\hyphen D^{\rb}_{\ovhol}(\bG_{\rmm,k}/K)$.
        \item[\textup{(iii)}] The convolution is associative, that is,
            we have a natural isomorphism $(\sM_1\ast\sM_2)\ast\sM_3\cong\sM_1\ast(\sM_2\ast\sM_3)$
            for all $\sM_1,\sM_2,\sM_3\in F\hyphen D^{\rb}_{\ovhol}(\bG_{\rmm,k}/K)$.
    \end{enumerate}
\end{proposition}

\begin{proof}
    This is an easy exercise of six functors.
    As an example, let us prove (i). 
    Let $\iota_1\colon\{1\}\hookrightarrow \bG_{\rmm,k}$ be the inclusion.
    Then, we have $\delta_1=\iota_{1,\plus}(\sO_{\{1\},\bQ})$, and because $\iota_1$ is proper
    we also have $\delta_1=\iota_{1,!}(\sO_{\{1\},\bQ})$.
    Then, the K\"unneth formula shows that 
    $\delta_1\boxtimes\sM\cong (\iota_1\times\id_{\bG_{\rmm}})_{?}(\sO_{\{1\},\bQ}\boxtimes\sM)$
    for each $?\in\{\plus, !\}$, where the exterior tensor product in the right-hand side
    is taken on $\{1\}\times\bG_{\rmm,k}$.
    Under the identification $\{1\}\times\bG_{\rmm,k}\cong\bG_{\rmm,k}$, 
    the object $\sO_{\{1\},\bQ}\boxtimes\sM$ is identified with $\sM$, and the product map $\mu$ is identified with the identity map,
    which shows that $\delta_1\ast\sM\cong\sM$.
    The isomorphism $\sM\ast\delta_1\cong\sM$ is proved in the same way.
\end{proof}

\begin{proposition}
    \label{prop:frobeniusofconvolution}
    Let $\sM, \sN$ be objects of $F\hyphen D^{\rb}_{\ovhol}(\bG_{\rmm,k}/K)$,
    let $x$ be an element of $k$
    and consider it as a $k$-valued point of $\bG_{\rmm,k}$.
    For each element $y$ of $k$, we denote by $\Tr(\Phi|\sM_y)$ \resp{$\Tr(\Phi|\sN_y)$}
    the Frobenius trace of $\sM$ \resp{$\sN$} at the $k$-valued point defined by $y$.
    Then, the Frobenius trace of $\sM\ast_!\sN$ at the $k$-valued point defined by $x$ equals
    \[
        \sum_{x_1,x_2\in k, x_1x_2=x}\Tr(\Phi|\sM_{x_1})\Tr(\Phi|\sN_{x_2}).
    \]
\end{proposition}

\begin{proof}
    Let $\mu\colon\bG_{\rmm,k}\times\bG_{\rmm,k}\to\bG_{\rmm,k}$ denote the multiplication map
    and let $\pr_i\colon\bG_{\rmm,k}\times\bG_{\rmm,k}\to\bG_{\rmm,k}$ denote the $i$-th projection
    for $i\in\{1,2\}$.
    Recall that $\sM\ast_!\sN=\mu_!(\sM\boxtimes\sN)\cong\mu_!(\pr_1^{\plus}\sM\otimes\pr_2^{\plus}\sN)$.
    Let $G_x$ denote the fiber product of $i_x\colon \Spec(k)\hookrightarrow\bG_{\rmm,k}$ and $\mu$,
    let $i_G$ denote the inclusion $G_x\hookrightarrow\bG_{\rmm,k}\times\bG_{\rmm,k}$, and
    let $f\colon G_x\to\Spec(k)$ denote the structure morphism.
    Then, by the base change theorem, we have
    \[
        i_x^{\plus}(\sM\ast_!\sN)\cong f_!i_{G_x}^{\plus}(\pr_1^{\plus}\sM\otimes\pr_2^{\plus}\sN).
    \]
    The Frobenius trace of this object at a closed point can be calculated by the trace formula \cite[A.4.1]{Abe13Langlands}.
    Each $k$-valued point of $G_x$ is of the form $(i_{x_1},i_{x_2})$, where $x_1, x_2\in k^{\times}$ with $x_1x_2=x$,
    and where $i_{x_j}$ is the $k$-valued point $\Spec(k)\hookrightarrow\bG_{\rmm,k}$ defined by $x_j$.
    For this $k$-valued point, we have
    \[
        (i_{x_1},i_{x_2})^{\plus}i_{G_x}^{\plus}(\pr_1^{\plus}\sM\otimes\pr_2^{\plus}\sN)
        \cong (i_{x_1}^{\plus}\sM)\otimes(i_{x_2}^{\plus}\sN)
        \cong (i_{x_1}^{\plus}\sM)\totimes(i_{x_2}^{\plus}\sN),
    \]
    where the last isomorphism follows from \cite[Proposition after 5.8]{Abe11}.
    This shows the claim.
\end{proof}

\subsection{Relation with Fourier Transforms.}
    In this subsection, we will show that the convolutions are ``symmetric'' generalizations of
    the geometric Fourier transform.

    In this subsection, $\inv\colon\bG_{\rmm,k}\to\bG_{\rmm,k}$ denotes the morphism defined by $x\mapsto x^{-1}$,
    $\pr_i\colon\bG_{\rmm,k}\times\bG_{\rmm,k}\to\bG_{\rmm,k}$ denotes the $i$-th projections for $i\in\{1,2\}$.

\begin{lemma}
    \label{lem:mcandintegral}
    For each objects $\sM, \sN\in F\hyphen D^{\rb}_{\ovhol}(\bG_{\rmm,k}/K)$, we have canonical isomorphisms
    \[
        \sM\ast_{\plus}\sN \cong \pr_{2,\plus}\big(\pr_1^{!}(\inv^{\ast}\sM)\widetilde{\otimes}\mu^{!}\sN\big)
        \quad\text{and}\quad
        \sM\ast_{!}\sN \cong \pr_{2,!}\big(\pr_1^{!}(\inv^{\ast}\sM)\widetilde{\otimes}\mu^{!}\sN\big).
    \]
\end{lemma}
\begin{proof}
    The proof goes as in \cite[(5.1.10)]{Katz90ESDE}.
    Let $\sigma\colon \bG_{\rmm,k}\times\bG_{\rmm,k}\to\bG_{\rmm,k}\times\bG_{\rmm,k}$ the morphism defined by $(x,y)\mapsto (x^{-1},xy)$.
    Since $\mu=\pr_2\circ\sigma$, $\sM\ast_{\plus}\sN$ \resp{$\sM\ast_{!}\sN$} is isomorphic to 
    $\pr_{2,\plus}\circ\sigma_{\plus}(\pr_1^{!}\sM\widetilde{\otimes}\pr_2^{!}\sN)$
    \resp{$\pr_{2,!}\circ\sigma_{\plus}(\pr_1^{!}\sM\widetilde{\otimes}\pr_2^{!}\sN)$ because $\sigma_{\plus}=\sigma_!$}.
    By using $\inv\circ\pr_1\circ\sigma=\pr_1$, we have for each $?\in\{\plus,!\}$
    \[
        \sM\ast_{?}\sN \cong \pr_{2,?}\circ\sigma_{\plus} \big(\sigma^{\ast}\pr_1^!(\inv^{\ast}\sM)\widetilde{\otimes}\pr_2^!\sN\big)\cong
        \pr_{2,?}\big(\pr_1^!(\inv^{\ast}\sM)\totimes\sigma_{\plus}(\pr_2^!\sN)\big),
    \]
    where the second isomorphism is the projection formula applied to $\sigma$.
    Since $\sigma\circ\sigma=\id$, we have a natural isomorphism $\sigma^{\ast}\cong\sigma_{\plus}$ of functors,
    which shows that $\sigma_{\plus}\circ\pr_2^!\cong (\pr_2\circ\sigma)^!=\mu^!$. This shows the claim.
\end{proof}

\begin{proposition}
    \label{prop:fourierandconv}
    Assume that $K$ contains a primitive $p$-th root of unity.
    We denote by $j\colon \bG_{\rmm,k}\hookrightarrow\bA^1_k$ the natural inclusion.
    Then, for each object $\sM$ of $F\hyphen D^{\rb}_{\ovhol}(\bG_{\rmm,k}/K)$, we have natural isomorphisms
    \[
        j^{\ast}\big(\FT_{\psi}(j_{\plus}\inv^{\ast}\sM)\big) \cong \sM\ast_{\plus} (j^{\ast}\sL_{\psi})[-1]
        \quad\text{and}\quad
        j^{\ast}\big(\FT_{\psi}(j_!\inv^{\ast}\sM)\big) \cong \sM\ast_{!} (j^{\ast}\sL_{\psi})[-1].
    \]
\end{proposition}
\begin{proof}
    The proof of the first isomorphism goes precisely as in the over-$\bC$ case \cite[5.2.3]{Katz90ESDE} and is omitted here.
    In the same way, we may construct an isomorphism
    \[
        j^{\ast}\bigg( \pr_{2,!}\big(\sL_{\mu,\psi}\totimes\pr_1^!j_!\inv^{\ast}\sM\big)\bigg)\cong\sM\ast_{!}(j^{\ast}\sL_{\psi}).
    \]
    Now, the isomorphism $\pr_{2,!}(\sL_{\mu,\psi}\totimes\pr_1^!\sN)\cong \FT_{\psi}(\sN)$
    (Proposition \ref{prop:miracleofft})
    applied to $\sN=j_!\inv^{\ast}\sM$ proves the second isomorphism.
\end{proof}

\section{$p$-adic hypergeometric $\sD$-modules.}

In this section, we introduce the $p$-adic hypergeometric $\sD$-modules in two ways;
one is given by explicit hypergeometric equations, and the other uses the multiplicative convolutions.
After that, we compare these two $\sD$-modules and investigate the properties of them.

\subsection{Hypergeometric differential operators.}

First, let us define a hypergeometric differential $\sD^{\dag}$-module on $\bG_{\rmm,k}$
by using hypergeometric differential operators.

\begin{definition}
    \label{def:hgoperators}
    Let $\pi$ be a non-zero element of $K$, and
    let $\alpha_1,\ldots,\alpha_m,\beta_1,\ldots,\beta_n$ be elements of $K$.
    We write the sequence $\alpha_1,\ldots,\alpha_m$ by $\balpha$ and $\beta_1,\ldots,\beta_n$ by $\bbeta$.
\begin{itemize}
    \item[\textup{(i)}] We define the hypergeometric operator $\Hyp_{\pi}(\balpha;\bbeta)=\Hyp_{\pi}(\alpha_1,\ldots,\alpha_m;\beta_1,\ldots,\beta_n)$ to be
    \[
        \Hyp_{\pi}(\balpha;\bbeta) \defeq \prod_{i=1}^m (x\partial-\alpha_i) - (-1)^{m+np}\pi^{m-n}x\prod_{j=1}^n(x\partial-\beta_j) 
    \]

\item[\textup{(ii)}] We define a $B_1(K)^{\dag}$-module $\sH_{\pi}(\balpha;\bbeta)=\sH_{\pi}(\alpha_1,\ldots,\alpha_m;\beta_1,\ldots,\beta_n)$ by
\[
    \sH_{\pi}(\balpha;\bbeta) \defeq B_1(K)^{\dag}/B_1(K)^{\dag}\Hyp_{\pi}(\balpha;\bbeta).
\]
\end{itemize}
\end{definition}

\begin{remark}
    Recall from \ref{para:DoverGm} that we are identifying the category of coherent $\sD_{\widehat{\bP^1_V},\bQ}^{\dag}(\pdag{\{0,\infty\}})$-modules
and that of coherent $B_1(K)^{\dag}$-modules.
Since $\sH_{\pi}(\balpha; \bbeta)$ is a coherent $B_1(K)^{\dag}$-module by definition, 
it is also regarded as a coherent $\sD_{\bP^1,\bQ}^{\dag}(\pdag{\{0,\infty\}})$-module.

However, the fact that $\sH_{\pi}(\balpha; \bbeta)$ is a coherent $\sD_{\widehat{\bP^1_V},\bQ}^{\dag}(\pdag{\{0,\infty\}})$-module does not mean that $\sH_{\pi}(\balpha; \bbeta)$ is coherent as a $\sD_{\widehat{\bP^1_V},\bQ}^{\dag}$-module
    (\cite[Remarque after Th\'eor\`eme 4.4.12]{BerthelotI}).
This is a reason why we cannot immediately conclude that it is an overholonomic $\sD_{\widehat{\bP^1_V},\bQ}^{\dag}$-module.
    We later show that it is in fact overholonomic if $\pi$ is the element associated with a non-trivial
	additive character and if $\alpha_i$'s and $\beta_j$'s are the elements associated with multiplicative characters 
	under a suitable condition.
\end{remark}

\begin{lemma}
    \label{lem:calcofhyp}
    Under the notation in Definition \ref{def:hgoperators}, $\sH_{\pi}(\balpha; \bbeta)$ has the following properties.
\begin{itemize}
    \item[\textup{(i)}] Let $\widetilde{\inv}\colon(\widehat{\bP^1_V},\{0,\infty\})\to(\widehat{\bP^1_V},\{0, \infty\})$ denote the morphism of d-couples defined by the inversion morphism $\overline{\inv}\colon\widehat{\bP^1_V}\to\widehat{\bP^1_V}$.
        (This morphism $\widetilde{\inv}$ realizes the inversion morphism $\inv\colon\bG_{\rmm,k}\to\bG_{\rmm,k}$.)
        Then, $\widetilde{\inv}^{\ast}\sH_{\pi}(\balpha; \bbeta)$ is isomorphic to $\sH_{(-1)^p\pi}(-\bbeta, -\balpha)$,
        where $-\balpha$ \resp{$-\bbeta$} denotes the sequence $-\alpha_1,\ldots,-\alpha_m$ \resp{$-\beta_1,\ldots,-\beta_n$}.
    \item[\textup{(ii)}] Let $\gamma$ be an element of $\frac{1}{q-1}\bZ$. Then, $\sH_{\pi}(\balpha;\bbeta)\otimes^{\dag}_{\sO_{\bP^1,\bQ}(\pdag{\{0,\infty\}})}\sK_{\gamma}$ is isomorphic to
        $\sH_{\pi}(\balpha+\gamma;\bbeta+\gamma)$, where $\balpha+\gamma$ \resp{$\bbeta+\gamma$} denotes the sequence
        $\alpha_1+\gamma,\ldots,\alpha_m+\gamma$ \resp{$\beta_1+\gamma,\ldots,\beta_n+\gamma$}.
\end{itemize}
\end{lemma}

\begin{proof}
    (i) Since $\widetilde{\inv}^{\ast}$ is identified with the base extension via $B_1(K)^{\dag}\to B_1(K)^{\dag}$ defined by 
    $x\mapsto x^{-1}$ and $\partial\mapsto -x^2\partial$, it suffices to show that the left ideal generated by
\[
    \prod_{i=1}^m(-x\partial-\alpha_i)-(-1)^{m+np}\pi^{m-n}x^{-1}\prod_{j=1}^n(-x\partial-\beta_j)
\]
equals that generated by $\Hyp_{(-1)^p\pi}(-\bbeta;-\balpha)$.
    This can be seen by a direct calculation:
    \begin{align*}
        & \prod_{i=1}^m(-x\partial-\alpha_i)-(-1)^{m+np}\pi^{m-n}x^{-1}\prod_{j=1}^n(-x\partial-\beta_j)\\
        =& (-1)^{m+np+n+1}\pi^{m-n}x^{-1}
        \left\{\prod_{j=1}^n(x\partial+\beta_j)-(-1)^{n+mp}\big((-1)^p\pi\big)^{n-m}x\prod_{i=1}^m(x\partial+\alpha_i)\right\}\\
        =& (-1)^{m+np+n+1}\pi^{m-n}x^{-1}\Hyp_{(-1)^p\pi}(-\bbeta;-\balpha).
    \end{align*}

    (ii)
    We know that $\sK_{\gamma}=B_1(K)^{\dag}/B_1(K)^{\dag}(x\partial-\gamma)$ is isomorphic to $\sO_{\widehat{\bP^1_V},\bQ}(\pdag{\{0,\infty\}})$ as
    an $\sO_{\widehat{\bP^1_V},\bQ}(\pdag{\{0,\infty\}})$-module. 
    Therefore, the functor $\text{-- }\otimes^{\dag}_{\sO_{\widehat{\bP^1_V},\bQ}(\pdag{\{0,\infty\})}}\sK_{\gamma}$ on the category of coherent $B_1(K)^{\dag}$-modules
    does not change the underlying $\sO_{\widehat{\bP^1_V},\bQ}(\pdag{\{0,\infty\}})$-module,
    and $\partial$ acts on the resulting $B_1(K)^{\dag}$-module as
    the action of $\partial + \gamma x^{-1}$ on the original $B_1(K)^{\dag}$-module.
    By this change of the action, $x\partial$ goes $x\partial-\gamma$, which proves the assertion.
\end{proof}

Later, we also have to consider the hypergeometric differential operator ``on $\bA^1_k$'', not only ``on $\bG_{\rmm,k}$''.
The following proposition describes the first essential relationships between them.

\begin{proposition}
    Let $\pi$ be a non-zero element of $K$ and let $\alpha_1,\ldots,\alpha_m,\beta_1,\ldots,\beta_n$ be elements of $K$.
    Let $\widetilde{j}\colon(\bP^1_k,\{0,\infty\})\hookrightarrow(\bP^1_k,\{\infty\})$ denote the morphism of d-couples defined by $\id_{\bP^1_k}$.
    (This morphism realizes the inclusion morphism $j\colon\bG_{\rmm,k}\hookrightarrow\bA^1_k$.)
    \begin{itemize}
        \item[\rm (i)] $\widetilde{j}^{\ast}\big(A_1(K)^{\dag}/A_1(K)^{\dag}\Hyp_{\pi}(\balpha;\bbeta)\big)$ is
            isomorphic to $\sH_{\pi}(\balpha;\bbeta)$.
        \item[\rm (ii)]
		Assume that $|\pi|=|p|^{1/(p-1)}$, that $\alpha_i$'s belong to $\frac{1}{q-1}\bZ\setminus\bZ$
		and that $\beta_j$'s belong to $\frac{1}{q-1}\bZ$.
           Then, the natural morphism
    \[
        A_1(K)^{\dag}/A_1(K)^{\dag}\Hyp_{\pi}(\balpha;\bbeta) \to \widetilde{j}_{\plus}\widetilde{j}^{\ast}\big(A_1(K)^{\dag}/A_1(K)^{\dag}\Hyp_{\pi}(\balpha;\bbeta)\big)
    \]
    of $A_1(K)^{\dag}$-modules is an isomorphism.
    \end{itemize}
    \label{prop:connectiontype}
\end{proposition}

\begin{proof}
    (i)
    The assertion follows from the fact that
    the functor $\widetilde{j}^{\ast}$ is exact on the category of coherent $\sD^{\dag}_{\widehat{\bP^1_V},\bQ}(\pdag{\{\infty\}})$-modules.

    (ii)
    Let $\sL$ denote the $A_1(K)^{\dag}$-module $A_1(K)^{\dag}/A_1(K)^{\dag}\Hyp_{\pi}(\balpha;\bbeta)$
    and let $\varphi\colon \sL\to\widetilde{j}_{\plus}\widetilde{j}^{\ast}\sL$ denote the natural morphism in question.
    In order to prove that $\varphi$ is an isomorphism,
    it suffices to prove that the left multiplication $l_x\colon\sL\to\sL$ by $x$ on $\sL$ is bijective.
    In fact, suppose that we have proved the bijectivity of $l_x$.
    Since the restriction of $\varphi$ (regarded as a morphism of $\sD^{\dag}_{\widehat{\bP^1_V},\bQ}(\pdag{\{\infty\}})$-modules)
    to $\widehat{\bP_V^1}\setminus\{0\}$ is an isomorphism,
    it suffices to show that the restriction of $\varphi$ on $\widehat{\bA^1_V}$ is an isomorphism.
    Let us denote by $\widetilde{i}\colon(\{0\},\emptyset)\to(\widehat{\bP^1_V},\{\infty\})$ the morphism of d-couples
    defined by the closed immersion $\bar{i}\colon\{0\}\hookrightarrow\widehat{\bP^1_V}$.
    Then, by the definition of extraordinary pull-back \cite[(1.1.6.1)]{Caro06CM}
    (cf. \cite[4.3.2]{Berthelot02Ast}),
    $\widetilde{i}^!\sL$ is the complex $\left[\sL\xrightarrow{l_x}\sL\right]$, where the target is placed at degree zero.
    Therefore, the bijectivity of $l_x$ is equivalent to $\widetilde{i}^!\sL=0$.
    Moreover, if we denote by $i'\colon \{0\}\hookrightarrow\bA^1_k$ the inclusion,
    then $\widetilde{i}^!\sL=0$ implies $i'^!(\sL|_{\widehat{\bA^1_V}})=0$.
    In the localization triangle \cite[(1.1.6.5)]{Caro06CM} in $D^{\rb}_{\coh}(\sD^{\dag}_{\widehat{\bA^1_V},\bQ})$,
    \[
        \bR\underline{\Gamma}^{\dag}_{\{0\}}\big(\sL|_{\widehat{\bA^1_V}}\big)\longrightarrow \sL|_{\widehat{\bA^1_V}}\longrightarrow \big(\widetilde{j}_{\plus}\widetilde{j}^{\ast}\sL\big)|_{\widehat{\bA^1_V}}\longrightarrow +1,
    \]
    the first term is isomorphic to $i'_{\plus}\circ i'^{!}\big(\sL|_{\widehat{\bA^1_V}}\big)$ \cite[Corollaire 3.4.7, Th\'eor\`eme 3.4.9 and the first point in Remarques 3.4.10]{Caro12}.
    This shows that the restriction of $\varphi$ on $\widehat{\bA^1_V}$ is an isomorphism, which concludes the proof.

    Now, we show that $l_x$ is bijective. We firstly work on the injectivity.
    Let $P, Q$ be elements of $A_1(K)^{\dag}$ that satisfies $xP=Q\Hyp_{\pi}(\balpha;\bbeta)$.
    We show that $Q\in xA_1(K)^{\dag}$; then, since $x$ is not a zero-divisor in $A_1(K)^{\dag}$,
    we get that $P\in A_1(K)^{\dag}\Hyp_{\pi}(\balpha;\bbeta)$ and the injectivity follows.
    In order to show that $Q\in xA_1(K)^{\dag}$, we may assume that $Q$ is of the form
    $Q=\sum_{l=0}^{\infty}c_l\partial^{[l]}$, where $c_l$'s are elements of $K$
    satisfying $\exists C>0, \exists\eta<1, \forall l, |c_l|<C\eta^l$.
    Then,
    because $\Hyp_{\pi}(\balpha;\bbeta)=\prod_i(x\partial-\alpha_i)-(-1)^{m+np}\pi^{m-n} x\prod_j(x\partial-\beta_j)$,
    and because $\partial^{[l]}x\equiv\partial^{[l-1]} \pmod{xA_1(K)^{\dag}}$, we have
    \begin{align*}
        Q\Hyp_{\pi}(\balpha;\bbeta)&\equiv 
        \sum_{l=0}^{\infty}c_l\prod_{i=1}^m(l-\alpha_i)\partial^{[l]}
        -(-1)^{m+np}\pi^{m-n}\sum_{l=1}^{\infty}c_l\prod_{j=1}^n(l-1-\beta_j)\partial^{[l-1]}\\ 
        & \hspace{200pt} \pmod{xA_1(K)^{\dag}}.
    \end{align*}
    By assumption, the left-hand side belongs to $xA_1(K)^{\dag}$, which shows that, for each $l$,
    \[
        c_l\prod_{i=1}^m(l-\alpha_i) =
        (-1)^{m+np}\pi^{m-n}c_{l+1}\prod_{j=1}^n(l-\beta_j).
    \]
    Fix a positive integer $l$ that exceeds all $\alpha_j$'s and $\beta_j$'s.
    Then, for each natural number $k$, we have
    \[
        c_{l+k}=(-1)^{k(m+np)}\pi^{-k(m-n)}\frac{\prod_{i=1}^m(l+k-1-\alpha_i)(l+k-2-\alpha_i)\dots(l-\alpha_i)}
        {\prod_{j=1}^n(l+k-1-\beta_j)(l+k-2-\beta_j)\dots(l-\beta_j)}c_l.
    \]
    Lemma \ref{lem:valofvp} below shows that $\left|(l+k-1-\beta_j)\ldots(l-\beta_j)\right|^{-1}\geq p^{k/(p-1)-1}k^{-1}$ and
    $\big|(l+k-1-\alpha_i)\ldots(l-\alpha_i)\big|\geq p^{-k/(p-1)}(q-1)^{-1}(l+k-1-\alpha_i)^{-1}$.
   By these inequalities and $|\pi|=p^{-1/(p-1)}$, we have
   \[
        |c_{l+k}| \geq  p^{-m}(q-1)^{-2m}k^{-n}\prod_{i=1}^m(l+k-1-\alpha_i)^{-1}|c_l|
    \]
    for each $k$. Since $|c_{l+k}|<C\eta^{l+k}$ for all $k$,
    we must have $|c_l|=0$, therefore $c_l=0$.
    Since $l$ supposed to be an arbitrary positive integer
    exceeding all $\alpha_j$'s and $\beta_j$'s, we get that $Q$ is a finite sum.
    Now, the recurrence relation for $c_l$ and the assumption that $\alpha_i$'s are not integers,
    we know that $Q=0$.

    Next, we show the surjectivity.
    Given $P\in A_1(K)^{\dag}$, we have to show that there exists $Q, R\in A_1(K)^{\dag}$
    such that $xQ=P+R\Hyp_{\pi}(\balpha;\bbeta)$.
    We may assume that $P$ is of the form $P=\sum_{l=0}^{\infty}c_l\partial^{[l]}$,
    where $c_l$'s are elements of $K$ satisfying $\exists C>0, \exists\eta<1, \forall l, |c_l|<C\eta^l$.
    We show that there exists $R\in A_1(K)^{\dag}$ of the form $R=\sum_{d=0}^{\infty}d_l\partial^{[l]}$
    that satisfies $P+R\Hyp_{\pi}(\balpha;\bbeta)\in xA_1(K)^{\dag}$.
    Let $l_0$ be the maximum of the elements in $\Set{\beta_j+1 | j\in\{1,\ldots,n\}}\cap\bZ_{\geq 0}$ if
    this set is not empty; if it is empty, then let $l_0=0$.
    Then, we may assume that $c_l=0$ if $l<l_0$ by the following reason.

    If $A_1(K)$ denotes the usual Weyl algebra with coefficients in $K$,
    then by our assumption on the parameters,
    the right multiplication by $\Hyp_{\pi}(\balpha;\bbeta)$ is bijective on $A_1(K)/xA_1(K)$ 
    \cite[2.9.4, (3)$\Rightarrow$(2)]{Katz90ESDE}.
    This shows that there exists $R'\in A_1(K)$ such that
    $\sum_{l=0}^{l_0-1}c_l\partial^{[l]}+R'\Hyp_{\pi}(\balpha;\bbeta)\in xA_1(K)$
    (The proof in the reference \cite{Katz90ESDE} is given over $\bC$, but it remains valid for all field of characteristic $0$).
    Now, we assume that $c_l=0$ if $l<l_0$.

    We put $d_l=0$ if $l<l_0$, and for each $s\geq 0$ we put
    \begin{equation}
        d_{l_0+s} = \sum_{t=s}^{\infty}(-1)^{(t-s)(m+np+1)}\pi^{(t-s)(m-n)}\frac{\prod_{j=1}^n(l_0+t-1-\beta_j)\ldots(l_0+s-\beta_j)}{\prod_{i=1}^m(l_0+t-\alpha_i)\ldots(l_0+s-\alpha_i)}c_{l_0+t};
        \label{eq:defofd}
    \end{equation}
    This infinite series actually converges; in fact, Lemma \ref{lem:valofvp} shows that
$\big|(l_0+t-1-\beta_j)\ldots(l_0+s-\beta_j)\big|\leq p^{-(t-s)(p-1)+1}(t-s)$ and that
\[
    \big|(l_0+t-\alpha_i)\ldots(l_0+s-\alpha_i)\big|^{-1}\leq p^{(t-s-1)/(p-1)}(q-1)(l_0+t-\alpha_i),
\]
and therefore the norm of each summand in the left-hand side is bounded from above by
    \[
        p^{m/(p-1)+n}(q-1)^mC\cdot (t-s)^m\prod_{i=1}^m(l_0+t-\alpha_i)\cdot\eta^{l_0+t}.
    \]
    This converges to $0$ as $t\to\infty$, which shows that $d_{l_0+s}$ is well-defined.

    Now, we put $R=\sum_{l=0}^{\infty}d_l\partial^{[l]}$. Then, by the bound calculated above,
    $R$ is an element of $A_1(K)^{\dag}$.
    In fact, we have $|d_{l_0+s}|\leq C'\max\big\{(t-s)^m\prod_i(l_0+t-\alpha_i)\eta^{l_0+t}\big\}$
    for a constant $C'>0$,
    where the max is taken for $t\geq s$.
    By choosing $\eta'$ satisfying $\eta>\eta'>1$, we have
    $|d_{l_0+s}|\leq C'\eta'^{l_0+t}$ for sufficiently large $s$, which shows that $R\in A_1(K)^{\dag}$.

    Finally, we show that $R$ satisfies $P+R\Hyp_{\pi}(\balpha;\bbeta)\in xA_1(K)^{\dag}$.
    This is equivalent to showing that
    \[
        d_l\prod_{i=1}^m(l-\alpha_i)-(-1)^{(m+np)}\pi^{m-n}d_{l+1}\prod_{j=1}^n(l-\beta_j)+c_l=0
    \]
    for all $l\geq 0$.
    It trivially holds if $l<l_0-1$ because $d_l=d_{l+1}=c_l=0$ in this case;
    it also holds if $l=l_0-1$ because $d_l=c_l=0$ and $l-\beta_j=0$ for some $j$;
    otherwise, we may check it directly by using (\ref{eq:defofd}).

    This completes the proof of the surjectivity of $l_x$, and therefore of the proposition.
\end{proof}

\begin{lemma}
    Let $l, N$ be natural numbers that satisfies $l\leq N$, and let $\alpha$ be an element of $\frac{1}{q-1}\bZ$.
    Then, we have
    \[
        \left|\prod_{s=l}^N(s-\alpha)\right|\leq p^{-(N-l+1)/(p-1)+1}(N-l+1)
    \]
    If $\alpha<l$, then we also have
    \[
        \left|\prod_{s=l}^N(s-\alpha)\right|\geq p^{-(N-l+1)/(p-1)}(q-1)^{-1}(N-\alpha)^{-1}
    \]
    \label{lem:valofvp}
\end{lemma}

\begin{proof}
    Let $m$ be a positive integer, and let $t_m$ be the number of $(s-\alpha)$'s for $s=l,\ldots,N$ that belongs to $p^m\bZ_{(p)}$:
    \[
        t_m \defeq \#\Set{ s\in\{l,\ldots,N\} | s-\alpha\in p^m\bZ_{(p)} }.
    \]
    Then, if $v_p$ denotes the $p$-adic valuation so that $v_p(p)=1$, we have
        $v_p\left(\prod_{s=l}^N(s-\alpha)\right) = \sum_{m=1}^{\infty} t_m$.
    Because $\alpha\in\bZ_{(p)}$, there is exactly one multiple of $p^m$ in every $p^m$ successive $(s-\alpha)$'s,
    and as a result we have $\floor*{\frac{N-l+1}{p^m}} \leq t_m\leq \floor*{\frac{N-l+1}{p^m}}+1$.
    Moreover, in case $\alpha<l$,
    we have $t_m=0$ unless $p^m\leq(q-1)(N-\alpha)$;
    in fact, $s-\alpha$ (for $s=l,\ldots,N$) is a multiple of $p^m$ if and only if so is the integer $(q-1)(s-\alpha)$.
    Now, since
    \[
        \frac{N-l+1}{p-1}-\log_p(N-l+1)-1\leq \sum_{m=1}^{\infty}\floor*{\frac{N-l+1}{p^m}} \leq\frac{N-l+1}{p-1},
    \]
    $v_p\left(\prod_{s=l}^N(s-\alpha)\right)=\sum_{m=1}^{\infty}t_m$ satisfies
    \[
        \sum_{m=1}^{\infty}t_m\geq \sum_{m=1}^{\infty}\floor*{\frac{N-l+1}{p^m}}\geq \frac{N-l+1}{p-1}-\log_p(N-l+1)-1
    \]
    in general. Moreover, if $\alpha<l$, then by the discussion above we have
    \[
        \displaystyle\sum_{m=1}^{\infty}t_m=\sum_{m=1}^{\log_p(q-1)(N-\alpha)}t_m
    \]
    and it satisfies
    \[
        \sum_{m=1}^{\log_p(q-1)(N-\alpha)}t_m
        \leq \sum_{m=1}^{\infty}\floor*{\frac{N-l+1}{p^m}}+\log_p(q-1)(N-\alpha)\leq\frac{N-l+1}{p-1}+\log_p(q-1)(N-\alpha).
    \]
    This shows the assertion.
\end{proof}

\subsection{Arithmetic hypergeometric $\sD$-modules and convolution}

In this section, firstly, we give another construction of arithmetic hypergeometric $\sD$-modules.
Under this definition, these $\sD$-modules are overholonomic and have Frobenius structures by nature.
Secondly, we compare these $\sD$-modules with the ones given in the previous subsection.
The comparison is the main part of this article.

\emph{In the remaining part of this article, we always assume that $K$ has a primitive $p$-th root of unity.}

\begin{definition}
    \label{def:defofh}
	Let $\psi$ be a non-trivial additive character on $k$, and
	let $\chi_1,\dots,\chi_m$ and $\rho_1,\dots,\rho_n$ be multiplicative characters on $k^{\times}$.
    Let $j\colon\bG_{\rmm,k}\hookrightarrow\bA^1_k$ denote the inclusion.
    We denote the sequence $\chi_1,\ldots,\chi_m$ by $\bchi$ and $\rho_1,\ldots,\rho_n$ by $\brho$; the empty sequence is denoted by $\emptyset$.
    Then, we define an object $\sHyp_{\psi,\plus}(\bchi;\brho)$ of $F\hyphen D_{\ovhol}^{\rb}(\bG_{\rmm,k}/K)$ as follows.
    
    (i) If $(m,n)=(0,0)$, then $\sHyp_{\psi,\plus}(\emptyset;\emptyset)\defeq \delta_1$.

    (ii) If $(m,n)=(1,0)$, then $\sHyp_{\psi,\plus}(\chi_1;\emptyset)\defeq j^{\ast}\sL_{\psi}\totimes\sK_{\chi_1}[2]$.

	(iii) If $(m,n)=(0,1)$, then $\sHyp_{\pi,\plus}(\emptyset;\rho_1)\defeq\inv^{\ast}\big(j^{\ast}\sL_{\psi^{-1}}\totimes\sK_{\rho_1^{-1}}\big)[2]$.

    (iv) Otherwise, $\sHyp_{\psi,\plus}(\bchi;\brho)$ is defined by
    \[
        \sHyp_{\psi,\plus}(\chi_1;\emptyset)\ast_{\plus}\ldots\ast_{\plus}\sHyp_{\psi,\plus}(\chi_m;\emptyset)
        \ast_{\plus} \sHyp_{\psi,\plus}(\emptyset;\rho_1)\ast_{\plus}\ldots\ast_{\plus}\sHyp_{\psi,\plus}(\emptyset;\rho_n).
    \]

    We also define an object $\sHyp_{\psi,!}(\bchi;\brho)$ of $F\hyphen D_{\ovhol}^{\rb}(\bG_{\rmm,k}/K)$
    in a similar way.
    Namely, we put $\sHyp_{\psi,!}(\bchi;\brho)\defeq\sHyp_{\psi,\plus}(\bchi;\brho)$ if $(m,n)\in\big\{(0,0), (1,0), (0,1)\big\}$, and
    otherwise $\sHyp_{\psi,!}(\bchi;\brho)$ is defined by (iv) above but $\ast_{\plus}$'s are replaced by $\ast_{!}$.
\end{definition}

\begin{remark}
    \label{rem:hypandh}
    If $(m,n)\in\big\{(0,0),(1,0),(0,1)\big\}$, then we immediately get a concrete description of $\sHyp_{\psi,\plus}(\bchi;\brho)$ (and therefore of $\sHyp_{\psi,!}(\bchi;\brho)$) as a $B_1(K)^{\dag}$-module:
    \begin{align*}
        \sHyp_{\psi,\plus}(\emptyset;\emptyset) & = B_1(K)^{\dag}/B_1(K)^{\dag}(1-x),\\
        \sHyp_{\psi,\plus}(\chi_1 ;\emptyset) & = \big(B_1(K)^{\dag}/B_1(K)^{\dag}(x\partial-\alpha_{\chi_1}+\pi_{\psi} x)\big)[1], \text{ and}\\
        \sHyp_{\psi,\plus}(\emptyset;\rho_1) & = \big(B_1(K)^{\dag}/B_1(K)^{\dag}(1-(-1)^p\pi_{\psi}^{-1}x(x\partial-\alpha_{\rho_1}))\big)[1].
    \end{align*}
    Here, $\alpha_{\chi_1}$ \resp{$\alpha_{\rho_1}$} is an element of $\frac{1}{q-1}\bZ$ satisfying
    $\chi_1(\xi)=\widetilde{\xi}^{(q-1)\alpha_{\chi_1}}$ \resp{$\rho_1(\xi)=\widetilde{\xi}^{(q-1)\alpha_{\rho_1}}$}.
    The first equation follows from the definition of the delta module,
    and the other two equations can be proved by the same calculation as we did in Lemma \ref{lem:calcofhyp}.
\end{remark}

This remark shows that, for these $(m,n)$'s, then $\sHyp_{\psi,\plus}(\bchi;\brho)$ and $\sHyp_{\psi,!}(\bchi;\brho)$ are
isomorphic to $\sH_{\pi_{\psi}}(\balpha;\bbeta)$ as $B_1(K)^{\dag}$-modules modulo a shift of degree.
The aim of this section is to generalize this fact to the case where $m$ and $n$ are larger.

Before doing that, we prove the similar formula for $\sHyp_{\psi}$'s as Lemma \ref{lem:calcofhyp}.

\begin{lemma}
    Under the notation in Definition \ref{def:defofh}, assume that $(m,n)\neq(0,0)$.
    Then, $\sHyp_{\psi,?}(\bchi;\brho)$ has the following properties where $?$ denotes $\plus$ or $!$.
    \label{lem:calcofh}
    \begin{itemize}
        \item[(i)] $\inv^{\ast}\sHyp_{\psi,?}(\bchi;\brho)$ is isomorphic to $\sHyp_{\psi^{-1},?}(\brho^{-1};\bchi^{-1})$, where $\brho^{-1}$ \resp{$\bchi^{-1}$} denotes the sequence $\rho_1^{-1},\dots,\rho_m^{-1}$ \resp{$\chi_1^{-1},\dots,\chi_n^{-1}$}.
        \item[(ii)] Let $\gamma$ be a character on $k^{\times}$. If $(m,n)\neq(0,0)$,
            then $\sHyp_{\psi,?}(\bchi;\brho)\totimes(\sK_{\gamma}[1])$
            is isomorphic to $\sHyp_{\psi,?}(\bchi\gamma;\brho\gamma)$,
            where $\bchi\gamma$ \resp{$\brho\gamma$} denotes the sequence $\chi_1\gamma,\dots,\chi_m\gamma$ \resp{$\rho_1\gamma,\dots,\rho_n\gamma$}.
    \end{itemize}
\end{lemma}

\begin{proof}
    (i) 
    If $(m,n)=(1,0)$ or $(m,n)=(0,1)$, then the assertion is obvious by definition.
    Otherwise, it follows from the formula $\inv^{\ast}(\sM\ast_{?}\sN)\cong(\inv^{\ast}\sM)\ast_{?}(\inv^{\ast}\sN)$, whose proof is immediate.

    (ii) 
    If $(m,n)=(1,0)$, then the claim directly follows from the formula $\sK_{\rho}\totimes\sK_{\chi}[1]\cong
    \big(\sK_{\rho}\otimes^{\dag}_{\sO(\pdag\{0,\infty\})}\sK_{\chi}[1]\big)[-1]\cong\sK_{\rho\chi}$.
    If $(m,n)=(0,1)$, then it suffices to use this formula and the formula $\inv^{\ast}\sK_{\chi^{-1}}\cong\sK_{\chi}$ twice.

    In order to prove the assertion in the general case,
    we first note that $\mu^!\sK_{\gamma}\cong\sK_{\gamma}\boxtimes\sK_{\gamma}[1]$.
    Therefore, for arbitrary objects $\sM$ and $\sN$ in $F\hyphen D^{\rb}_{\ovhol}(\bG_{\rmm}/K)$, we have
    \begin{align*}
        (\sM\totimes\sK_{\gamma}[1])\ast_{\plus}(\sN\totimes\sK_{\gamma}[1]) &\cong\mu_{\plus}\big((\sM\boxtimes\sN)\totimes(\sK_{\gamma}\boxtimes\sK_{\gamma})\big)[2]\\
        & \cong (\sM\ast_{\plus}\sN)\otimes\sK_{\gamma}[1].
    \end{align*}
    This fact and the assertion for $(m,n)=(1,0),(0,1)$ finish the proof.
\end{proof}

\begin{lemma}
    Let $\psi$ be a non-trivial additive character on $k$,
    and let $\chi_1,\dots,\chi_m$ and $\rho_1,\dots,\rho_n$ be characters on $k^{\times}$.
    Let $k'$ be a finite extension of $k$.
    We put $\psi'\defeq\psi\circ\Tr_{k'/k}$, $\chi'_i\defeq\chi_i\circ\Norm_{k'/k}$ and $\rho'_j\defeq\rho_j\circ\Norm_{k'/k}$.
    Moreover, we denote by $\bchi$ \resp{$\brho$} the sequence $\chi_1,\dots,\chi_m$ \resp{$\rho_1,\dots,\rho_n$}
    and by $\bchi'$ \resp{$\brho'$} the sequence $\chi'_1,\dots,\chi'_m$ \resp{$\rho'_1,\dots,\rho'_n$}.
    Then, we have
    \[
        \iota_{k'/k}\big(\sHyp_{\psi,\plus}(\bchi,\brho)\big)=\sHyp_{\psi',\plus}(\bchi',\brho')
    \]
    and
    \[
        \iota_{k'/k}\big(\sHyp_{\psi,!}(\bchi,\brho)\big)=\sHyp_{\psi',\plus}(\bchi',\brho').
    \]
    \label{lem:scalarextofhyp}
\end{lemma}

\begin{proof}
    If $(m,n)=(0,0)$, the assertion is obvious.
    If $(m,n)\neq (0,0)$, then it follows from the description of scalar extension of Dwork modules and 
    of Kummer modules (Subsection 1.3)
    and from the fact that the scalar extension commutes with the six-functor formalism.
\end{proof}

\begin{theorem}
    Let $\psi$ be a non-trivial additive character on $k$, and
    let $\chi_1,\dots,\chi_m$, $\rho_1,\ldots,\rho_n$ be characters on $k^{\times}$.
    Let $\balpha$ \resp{$\bbeta$} denote the sequence
    $\alpha_1,\dots,\alpha_m$ \resp{$\beta_1,\dots,\beta_n$},
    where $\alpha_i$ \resp{$\beta_j$} is an element of $\frac{1}{q-1}\bZ$ that satisfies
    $\chi_i(\xi)=\widetilde{\xi}^{(q-1)\alpha_i}$ \resp{$\rho_j(\xi)=\widetilde{\xi}^{(q-1)\beta_j}$}. 
    We assume that $(m,n)\neq(0,0)$ and that $\chi_i\neq\rho_j$ for any $i, j$.
    Then, we have an isomorphism
    \[
        \sHyp_{\psi,\plus}(\bchi;\brho) \cong \sH_{\pi_{\psi}}(\balpha;\bbeta)[m+n].
    \]
    as $\sD^{\dag}_{\widehat{\bP^1_V},\bQ}(\pdag\{0,\infty\})$-modules.
    \label{thm:hypandconv}
\end{theorem}

\begin{proof}
    We prove it by induction on $m+n$.
    The proof in the case $(m,n)=(1,0), (0,1)$ is already explained in Remark \ref{rem:hypandh}.

    First, by Lemma \ref{lem:calcofhyp} (i) and Lemma \ref{lem:calcofh} (i), we may assume that $m>0$.
    Moreover, by Lemma \ref{lem:calcofhyp} (ii) and Lemma \ref{lem:calcofh} (ii), we may assume that $\chi_1=\boldsymbol{1}$ and $\alpha_1=0$.
    Now, note that $\sHyp_{\psi,\plus}(\boldsymbol{1};\emptyset)=j^{\ast}\sL_{\pi}\totimes\sK_{\boldsymbol{1}}[2]=j^{\ast}\sL_{\pi}[1]$.
    It therefore suffices to show the isomorphism $\sHyp_{\psi,\plus}(\bchi';\brho)\ast_{\plus} (j^{\ast}\sL_{\psi}[1])\cong \sH_{\pi_{\psi}}(\balpha;\bbeta)[1]$ of $B_1(K)^{\dag}$-modules,
    where $\bchi'$ denotes the sequence $\chi_2,\ldots,\chi_m$.

    Firstly, Proposition \ref{prop:fourierandconv} shows that 
    \[
        \sHyp_{\psi,\plus}(\bchi';\brho)\ast_{\plus}(j^{\ast}\sL_{\psi}[1])
        \cong j^{\ast}\big(\FT_{\psi}(j_{\plus}\inv^{\ast}\sHyp_{\psi,\plus}(\bchi';\brho))\big).
    \]
    By Lemma \ref{lem:calcofh} (i), we have
    \[
        \inv^{\ast}\sHyp_{\psi,\plus}(\bchi';\brho) \cong\sHyp_{\psi^{-1},\plus}(\brho^{-1};\bchi'^{-1}),
    \]
    and by induction hypothesis, it is isomorphic to
    \[
        B_1(K)^{\dag}/B_1(K)^{\dag}\Hyp_{\pi_{\psi^{-1}}}(-\bbeta-1;-\balpha'-1)[m+n-1]
    \]
    as a $B_1(K)^{\dag}$-module.
    Moreover, by Proposition \ref{prop:connectiontype} (i), it is isomorphic to
    \[
        \widetilde{j}^{\ast}\left(A_1(K)^{\dag}/A_1(K)^{\dag}\Hyp_{\pi_{\psi^{-1}}}(-\bbeta-1;-\balpha'-1)\right)[m+n-1].
    \]
    Since there are no integers in the sequence $-\bbeta-1$ by our assumption, we conclude by using Proposition \ref{prop:connectiontype} (ii) that 
    \[
        j_{\plus}\inv^{\ast}\sHyp_{\psi,\plus}(\bchi';\brho)[m+n]\cong
        A_1(K)^{\dag}/A_1(K)^{\dag} \Hyp_{\pi_{\psi^{-1}}}(-\bbeta-1;-\balpha'-1)[m+n-1].
    \]
    Now, we may compute the Fourier transform of this object by using Proposition \ref{prop:geomftandnaiveft}
    and the formula $\pi_{\psi^{-1}}=(-1)^p\pi_{\psi}$, and the result is $A_1(K)^{\dag}/A_1(K)^{\dag}\Hyp_{\pi_{\psi}}(\balpha;\bbeta)[m+n]$.
    In fact, since under the notation in Proposition \ref{prop:globalsectionofft} we have $\varphi_{\pi_{\psi}}(x\partial)=-\partial x=-x \partial - 1$,
    \begin{align*}
        & \varphi_{\pi_{\psi}}\Hyp_{\pi_{\psi^{-1}}}(-\bbeta-1;-\balpha'-1) \\
        = & \varphi_{\pi_\psi}\left(\prod_{j=1}^n(x\partial+\beta_j+1)-(-1)^{n+(m-1)p}\big((-1)^p\pi_{\psi}\big)^{n-m+1}x\prod_{i=2}^m(x\partial+\alpha_i+1)\right)\\
        = &  \prod_{j=1}^n(-x\partial+\beta_j)-(-1)^{n(p-1)}\pi_{\psi}^{n-m+1}\frac{-\partial}{\pi_{\psi}}\prod_{i=2}^m( -x\partial+\alpha_i)\\
        = & \frac{(-1)^{n(p-1)+m-1}}{\pi_{\psi}^{m-n}x}\left\{x\partial\prod_{i=2}^m(x\partial-\alpha_i)-(-1)^{m+np}\pi_{\psi}^{m-n}\prod_{j=1}^n(x\partial-\beta_j)\right\}.
    \end{align*}
    Again by using Proposition \ref{prop:connectiontype} (i), it shows the assertion.
\end{proof}

\begin{corollary}
    Let $\pi$ be an element of $K$ that satisfies $\pi^{q-1}=(-p)^{(q-1)/(p-1)}$,
    and let $\alpha_1,\dots,\alpha_m,\beta_1,\dots,\beta_n$ be elements of $\frac{1}{q-1}\bZ$
    that satisfy $\alpha_i-\beta_j\not\in\bZ$ for any $i,j$.
    Then, the $B_1(K)^{\dag}$-module $\sH_{\pi}(\balpha; \bbeta)$ and
    the $A_1(K)^{\dag}$-module $A_1(K)^{\dag}/\Hyp_{\pi}(\balpha;\bbeta)A_1(K)^{\dag}$ are overholonomic and have a Frobenius structure.
    \label{cor:overholonomic}
\end{corollary}
\begin{proof}
    By assumption and Theorem \ref{thm:hypandconv},
    $\sH_{\pi}(\balpha;\bbeta)$ is isomorphic to $\sHyp_{\psi,\plus}(\bchi;\brho)$
    for some $\psi$, $\chi_i$'s and $\rho_j$'s satisfying the assumptions of Theorem \ref{thm:hypandconv}.
    Therefore, the assertion for $\sH_{\pi}(\balpha; \bbeta)$ now follows.
    The assertion for $A_1(K)^{\dag}/A_1(K)^{\dag}\Hyp_{\pi}(\balpha;\bbeta)$
    follows from the theorem together with Proposition \ref{prop:connectiontype} (i), (ii).
\end{proof}

\subsection{Comparison of two hypergeometric $\sD$-modules associated with characters}

In Definition \ref{def:defofh}, we defined two versions of hypergeometric $\sD$-modules
associated with characters;
one uses  $\ast_{!}$ and the other uses $\ast_{\plus}$.
The goal of this subsection is to prove that these are naturally isomorphic to each other under the hypothesis in Theorem \ref{thm:hypandconv}.

We begin with a variant of Proposition \ref{prop:connectiontype} (ii).

\begin{proposition}
    Let $\alpha_1,\ldots,\alpha_m,\beta_1,\ldots,\beta_n$ be elements of $\frac{1}{q-1}\bZ$.
    Assume that $\alpha_i\not\in\bZ$ for any $i$ and moreover that $\alpha_i-\beta_j\not\in\bZ$ for any $i,j$.
    Let $j\colon\bG_{\rmm}\hookrightarrow\bA^1$ be the inclusion.
    Then, the natural morphism
    \[
        j_!j^{\ast}\big(A_1(K)^{\dag}/A_1(K)^{\dag}\Hyp_{\pi}(\balpha;\bbeta)\big)\longrightarrow 
        A_1(K)^{\dag}/A_1(K)^{\dag}\Hyp_{\pi}(\balpha;\bbeta)
    \]
    is an isomorphism.
    \label{prop:connectiontype2}
\end{proposition}

\begin{proof}
    We put $\sL\defeq A_1(K)^{\dag}/A_1(K)^{\dag}\Hyp_{\pi}(\balpha;\bbeta)$;
    it is overholonomic and has a Frobenius structure by Corollary \ref{cor:overholonomic}.
    By the localization triangle $j_!j^{\ast}\to\id\to i_{\plus}i^{\plus}\to$ \cite[(3.1.9.1)]{AbeMarmora},
    it suffices to prove that $i^{\plus}\sL=0$.
    Since $i^{\plus}\sL$ is isomorphic to the complex $\left[\sL^{\an}\xrightarrow{\partial\cdot} \sL^{\an}\right]$
    \cite[5.1.3]{Crew12CM}, where $\sL^{\an}$ denotes the analytification of $\sL$,
    it suffices to prove that $\partial$ bijectively acts on $\sL^{\an}$.
    Moreover, since the kernel and the cokernel of this morphism have the same dimension
    \cite[3.1.10 and Remark after that]{AbeMarmora},
    it only remains to show that it is injective.

    Recall that the ring $\sD^{\an}$ of analytic differential operators \cite[4.1]{Crew12CM} is described as
    \[
        \sD^{\an}=\Set{\sum_{k=0}^{\infty}a_k\partial^{[k]} | a_k\in A[0,1[\text{ and } \exists \eta<1, \forall r<1, \exists C_r>0, |a_k|_r\leq C_r\eta^k}.
    \]
    Here, $A[0,1[$ denotes the ring of analytic functions on the open unit disk, and $|\cdot|_r$ denotes
        the $r$-Gauss norm on $A[0,1[$.
        Since we have $\sL^{\an}=\sD^{\an}/\sD^{\an}\Hyp_{\pi}(\balpha;\bbeta)$,
        in order to show the injectivity of $\partial\cdot$, it suffices to prove the following property:
        if $P, Q\in \sD^{\an}$ satisfies $\partial P=Q\Hyp_{\pi}(\balpha;\bbeta)$,
    then $Q$ belongs to $\partial \sD^{\an}$.
    We may assume that $Q\in A[0,1[$. We put $Q=\sum c_lx^l$. Since
    \[
        Q\Hyp_{\pi}(\balpha;\bbeta)\equiv \sum_{l=0}^{\infty}\left(\prod_{i=1}^m(-l-1-\alpha_i)x^l-(-1)^{m+np}\pi^{m-n}\prod_{j=1}^n(-l-2-\beta_j)x^{l+1}\right)c_l
    \]
    modulo $\partial\sD^{\an}$, the assumption that $\alpha_i\not\in\bZ$ shows inductively that
    all $c_l$'s are zero.
\end{proof}

\begin{corollary}
    Under the situation in the previous proposition, the natural morphism
    \[
    j_{!}\sH_{\pi}(\balpha; \bbeta)\longrightarrow j_{\plus}\sH_{\pi}(\balpha;\bbeta)
    \]
    is an isomorphism.
    \label{cor:middleext}
\end{corollary}

\begin{proof}
   Combine Proposition \ref{prop:connectiontype} and Proposition \ref{prop:connectiontype2}.
\end{proof}

Now, we get the following proposition.

\begin{proposition}
    Let $\psi$ be a non-trivial additive character on $k$,
    let $\chi_1,\ldots,\chi_m$ and $\rho_1,\ldots,\rho_n$ be characters on $k^{\times}$
    and assume that $\chi_i\neq\rho_j$ for any $i,j$.
    Then, the natural morphism
    \[
        \sHyp_{\psi,!}(\bchi; \brho)\longrightarrow\sHyp_{\psi,\plus}(\bchi; \brho)
    \]
    in $F\hyphen D^{\rb}_{\ovhol}(\bG_{\rmm,k}/K)$ is an isomorphism.
    \label{prop:twohypareisom}
\end{proposition}

\begin{proof}
    If $(m,n)=(0,0),(1,0),(0,1)$, the assertion is obvious by definition.
    Let us prove the assertion by induction on $m+n$.
    As in the proof of Theorem \ref{thm:hypandconv}, we may assume that $m>0$ and that $\chi_1$ is trivial.
    Then, (again by the similar argument as in the proof of Theorem \ref{thm:hypandconv},) the claim reduces to showing that
    \[
        j^{\ast}\big(\FT_{\psi}(j_{!}\sHyp_{\pi_{\psi^{-1}},!}(\brho^{-1};\bchi'^{-1}))\big)
        \longrightarrow j^{\ast}\big(\FT_{\psi}(j_{\plus}\sHyp_{\pi_{\psi},\plus}(\brho^{-1};\bchi'^{-1}))\big)
    \]
    is an isomorphism.
    By induction hypothesis, the natural morphism
    \[
        \sHyp_{\pi_{\psi^{-1}},!}(\brho^{-1};\bchi'^{-1})\longrightarrow\sHyp_{\pi_{\psi^{-1}},\plus}(\brho^{-1};\bchi'^{-1})
    \]
    is an isomorphism, and moreover Theorem \ref{thm:hypandconv} shows that both are isomorphic to
    $\sH_{\pi}(\balpha;\bbeta)$ for some $\pi$, $\alpha_i$'s and $\beta_j$'s satisfying
    the assumption of Proposition \ref{prop:connectiontype2}.
    Therefore, the morphism
    \[
        j_{!}\sHyp_{\pi_{\psi^{-1}},!}(\brho^{-1};\bchi'^{-1})\longrightarrow j_{\plus}\sHyp_{\pi_{\psi^{-1}},\plus}(\brho^{-1};\bchi'^{-1})
    \]
    is an isomorphism because the underlying morphism in $D^{\rb}_{\ovhol}(\bA^1_k/K)$ is an isomorphism by the previous corollary.
    This shows the assertion.
\end{proof}

\section{Hypergeometric Isocrystals.}

In this section, we study the isocrystals (in the classical sense)
defined by hypergeometric differential operators.
The first subsection compares it with the arithmetic hypergeometric $\sD$-modules in the previous section,
and as a result, proves that the latter is has a structure of overconvergent $F$-isocrystals.
In the second subsection, we discuss the irreducibility and a characterization of these overconvergent $F$-isocrystals.

Recall that we are always assuming that $K$ has a primitive $p$-th root of unity.

\subsection{Overconvergence of hypergeometric isocrystals.}

Throughout this subsection, we fix a $\pi\in K^{\times}$ with $\pi^{q-1}=(-p)^{(q-1)/(p-1)}$,
fix elements $\alpha_1,\dots,\alpha_m$ and $\beta_1,\dots,\beta_n$ in $\frac{1}{q-1}\bZ$
such that $\alpha_i-\beta_j\not\in\bZ$ for any $i,j$, and assume that $(m,n)\neq(0,0)$.
We put $r\defeq\max\{m,n\}$.
We denote by $T$ the divisor $\{0,\infty\}$ of $\bP^1_k$ if $m\neq n$ and $\{0,1,\infty\}$ if $m=n$,
and we put $\sX\defeq \widehat{\bP^1_V}\setminus T$ and $X\defeq\bP^1_k\setminus T$.
In this subsection, $h_l(x)\in K[x] (l\in\bN)$ denote the polynomial defined by
\[
    \Hyp_{\pi}(\balpha;\bbeta)=\sum_{l=0}^{\infty}h_l(x)\partial^l.
\]

\begin{lemma}
    \label{lem:coefficientofhgoperator}
    We have
    \[
        h_r(x) = \begin{cases} x^m & \text{if }m>n,\\ x^m\big(1-(-1)^{n(p+1)}x\big) & \text{if }m=n,\\ -(-1)^{m+np}\pi^{m-n}x^{n+1} & \text{if }m<n.\end{cases}
    \]
    If $l>r$, then $h_l(x)=0$.
    \label{lem:hr}
\end{lemma}
\begin{proof}
    $(x\partial)^l$ is the sum of $x^l\partial^l$ and a differential operator of differential order $<l$.
    The assertion follows from this fact.
\end{proof}

This lemma shows that $\frac{h_l(x)}{h_r(x)}$'s are global sections of $\sO_{\sX_K}$.
With this in mind, we define the hypergeometric isocrystals $H_{\pi}(\balpha;\bbeta)$ on $X$ as follows.

\begin{definition}
    The hypergeometric isocrystal $H_{\pi}(\balpha;\bbeta)$ is 
    the free $\sO_{\sX_K}$-module $(\sO_{\sX_K})^{\oplus r}$ equipped with the connection
    \[
        \nabla = \begin{pmatrix}
            0 & 1 & & & \\ & 0 & 1 & & \\ & & & & \\ & & & 0 & 1 \\ -\frac{h_0(x)}{h_r(x)} & -\frac{h_1(x)}{h_r(x)} & \dots & \dots & -\frac{h_{r-1}(x)}{h_r(x)}
        \end{pmatrix} \otimes dx.
    \]
\end{definition}

The following theorem, which relates the isocrystal $H_{\pi}(\balpha;\bbeta)$ and the arithmetic
$\sD$-module $\sH_{\pi}(\balpha;\bbeta)$, is the goal of this subsection.
For the convenience of the reader, we repeat the hypothesis in this subsection.

\begin{theorem}
    Let $\pi$ be an element of $K^{\times}$ satisfying $\pi^{q-1}=(-p)^{(q-1)/(p-1)}$,
    and let $\alpha_1,\dots,\alpha_m$ and $\beta_1,\dots,\beta_n$ be elements of $\frac{1}{q-1}\bZ$
    satisfying $\alpha_i-\beta_j$ for any $i,j$.
    Then, we have an isomorphism
    \[
        \sp_{\plus}H_{\pi}(\balpha;\bbeta)\cong \sH_{\pi}(\balpha;\bbeta).
    \]
    In particular, $H_{\pi}(\balpha;\bbeta)$ has a structure of a convergent $F$-isocrystal on $X$
    overconvergent along $T$.
    \label{thm:ocFisoc}
\end{theorem}

In the course of the proof, we need the theory of hypergeometric functions over finite fields
(or hypergeometric sums) \cite[(8.2.7)]{Katz90ESDE} and the $\ell$-adic hypergeometric sheaves \cite[Theorem 8.4.2]{Katz90ESDE}.

\begin{definition}
(i) Let $m, n$ be natural numbers and let $t$ be an element of $k$.
We define the hypersurface $V(m,n,t)$ of $(\bG_{\rmm,k})^{m+n}$ by the equation
\[
    \prod_{i=1}^m x_i \cdot \prod_{j=1}^n y_j^{-1} = t,
\]
where $(x_1,\dots,x_m,y_1,\dots,y_n)$ is the coordinate of $(\bG_{\rmm,k})^{m+n}$.

(ii) Let $\psi$ be a non-trivial additive character on $k$,
and let $\chi_1,\ldots,\chi_m,\rho_1,\ldots,\rho_n$ be multiplicative characters on $k^{\times}$.
We denote the sequence $\chi_1,\dots,\chi_m$ \resp{$\rho_1,\dots,\rho_n$}
by $\bchi$ \resp{$\brho$}.
Then, the function $\Hyp_{\psi}(\bchi; \brho)$ on $k^{\times}$ is defined by
\[
    \Hyp_{\psi}(\bchi;\brho)(t) \defeq \sum\psi\left(\sum_{i=1}^m x_i-\sum_{j=1}^n y_j\right)\cdot\prod_{i=1}^m\chi_i(x_i)\cdot
    \prod_{j=1}^n\rho_j^{-1}(y_j),
\]
where the sum runs over $(x_1,\dots,x_m,y_1,\dots,y_n)\in V(m,n,t)$.
\end{definition}

\begin{proposition}[{\cite[(8.2.7), Theorem 8.4.2]{Katz90ESDE}}]
    Under the notation in Definition 4.1.4 (ii), assume that $\chi_i\neq\rho_j$ for each $i,j$.
    Then, there exists a smooth $\ell$-adic sheaf $\sH^{\ell}_{\psi,!}(\bchi;\brho)$ of
    rank $r=\max\{m,n\}$ on $X$,
    whose Frobenius trace at a closed point $x$ of degree $h$ equals
    \[
        (-1)^{m+n+1}\Hyp_{\psi\circ\Tr_{\kappa(x)/k}}(\bchi\circ\Norm_{\kappa(x)/k};\brho\circ\Norm_{\kappa(x)/k}).
    \]
\end{proposition}

We refer the smooth $\ell$-adic sheaf $\sH^{\ell}_{\psi,!}(\bchi;\brho)$ in the previous proposition
as ``the $\ell$-adic hypergeometric sheaf''.

We are now ready for the proof of Theorem \ref{thm:ocFisoc}.
We start by calculating the Frobenius trace of $\sHyp_{\psi,!}(\bchi;\brho)$.

\begin{proposition}
    Let $\psi$ be a non-trivial additive character on $k$, and
    let $\chi_1,\dots,\chi_m$ and $\rho_1,\dots,\rho_n$ be characters on $k^{\times}$.
    Let $x$ be a $k$-rational point of $\bG_{\rmm,k}$.
    Then, the Frobenius trace of $\sHyp_{\psi,!}(\bchi;\brho)$ at $x$ equals
    $q^{2(m+n)}\Hyp_{\psi}(\bchi;\brho)(x)$.
    \label{prop:frobeniustrace}
\end{proposition}

\begin{proof}
    If $(m,n)=(1,0)$, Proposition \ref{prop:propertiesofDwork} (iii) and \ref{prop:propertiesofKummer} (iii)
    show that the Frobenius trace of $\sHyp_{\psi,!}(\chi;\emptyset)$ at $x$ equals $q^2\psi(x)\chi(x)$,
    which coincides with $q^2\Hyp_{\psi}(\chi;\emptyset)(x)$.
    The case where $(m,n)=(0,1)$ can be proved similarly.

    In general, for any sequences $\bchi$ \resp{$\bchi',\brho,\brho'$} of characters on $k^{\times}$, we have
    \[
        \Hyp_{\psi}(\bchi,\bchi';\brho,\brho')(t) = \sum_{t_1t_2=t}\Hyp_{\psi}(\bchi;\brho)(t_1)\Hyp_{\psi}(\bchi',\brho')(t_2).
    \]
    The claim is proved by this formula, Proposition \ref{prop:frobeniusofconvolution} and the previous two cases.
\end{proof}

\begin{lemma}
    \label{lem:dimofstalks}
    Let $T'$ be a divisor of $\bP^1_k$ that includes $T$,
    put $\sY\defeq\widehat{\bP^1_V}\setminus T'$, $Y\defeq \bP^1_k\setminus T'$,
    and assume that $\sH_{\pi}(\balpha;\bbeta)|_{\sY}$ is a coherent $\sO_{\sY,\bQ}$-module.
    Then, $\sH_{\pi}(\balpha;\bbeta)|_{\sY}$ is a free $\sO_{\sY,\bQ}$-module of rank $r$.
\end{lemma}

\begin{proof}
    By assumption, there exists a convergent isocrystal $E$ on $Y$ satisfying $\sp_{\plus}E\cong\sH_{\pi}(\balpha;\bbeta)|_{\sY}$.
    Since $E$ is free as an $\sO_{\sY_K}$-module (being locally free over the principal ideal ring $\Gamma(\sY_K, \sO_{\sY_K})$), $\sH_{\pi}(\balpha;\bbeta)|_{\sY}$ is also free
    as an $\sO_{\sY,\bQ}$-module.
    Therefore, it remains to prove that its rank equals $r$.

    Let $x$ be a closed point of $Y$ with residue field $\kappa(x)$.
    Then, $i_x^{\plus}\iota_{\kappa(x)/k}\sH_{\pi}(\balpha;\bbeta)$ is concentrated at degree $-1$
    \cite[(1.2.14.1)]{AbeCaro};
    for proving the lemma, it suffices to show that the $(-1)$-st cohomology is of dimension $r$.

    By the scalar extension, we may assume that $x$ is a $k$-rational point.
    Let $\psi$ \resp{$\chi_i$'s and $\rho_j$'s} be a non-trivial character on $k$ \resp{characters on $k^{\times}$}
    corresponding to $\pi$ \resp{$\alpha_i$'s and $\beta_j$'s},
    so that $\sH_{\pi}(\balpha;\bbeta)\cong \sHyp_{\psi,\plus}(\bchi;\brho)[-(m+n)]$.
    For each positive integer $h$, let $k_h$ be a field extension of $k$ of degree $h$.
    By using the previous proposition, we see just as in \cite[(8.2.7)]{Katz90ESDE} that the Frobenius trace 
    at $x$ of $\iota_{k_h/k}\big(\sHyp_{\psi,!}(\bchi,\brho)\big)[-(m+n)]$ equals
    \begin{equation}
        (-1)^{m+n}(q^h)^{2(m+n)}\Hyp_{\psi\circ\Tr_{k_h/k}}(\bchi\circ\Norm_{k_h/k};\brho\circ\Norm_{k_h/k}).
        \label{eq:traceatx}
    \end{equation}
    The assumption and the purity \cite[(1.2.14.1)]{AbeCaro} shows that
    the trace of the Frobenius action on this $(-1)$-st cohomology is the minus of (\ref{eq:traceatx}).
    Since it equals the Frobenius trace of the smooth $\ell$-adic sheaf
    $\sH^{\ell}_{\psi,!}(\bchi;\brho)$ of rank $r$ at $x$,
    there exist ($q$-Weil) numbers $\gamma_1,\dots,\gamma_r$ independent of $h$ satisfying
    \[
        (-1)^{m+n+1}(q^h)^{2(m+n)}\Hyp_{\psi\circ\Tr_{k_h/k}}(\bchi\circ\Norm_{k_h/k};\brho\circ\Norm_{k_h/k})
        = \gamma_1^h+\gamma_2^h+\dots+\gamma_r^h.
    \]
    Since the left-hand side is the trace of the $h$-th iterate of the Frobenius action on 
    the $(-1)$-st cohomology of $i_x^{\plus}\sHyp_{\psi,\plus}(\bchi;\brho)[-(m+n)]$,
    this action has exactly $r$ eigenvalues (Note that there are no zero eigenvalue.)
    This shows that its dimension equals $r$.
\end{proof}

\begin{lemma}
    \label{lem:phiisisom}
    Let $T'$ be a divisor of $\bP^1_k$ that includes $T$, and put $\sY\defeq\widehat{\bP^1_V}\setminus T'$.
    Assume that $\sH_{\pi}(\balpha;\bbeta)|_{\sY}$ is a coherent $\sO_{\sY,\bQ}$-module.
    Then, the composition
    \[
        \sD_{\sY,\bQ}\longrightarrow \sD_{\sY,\bQ}^{\dag} \longrightarrow \sH_{\pi}(\balpha;\bbeta)|_{\sY}
    \]
    induces an isomorphism
    \[
        \varphi\colon \sD_{\sY,\bQ}/\sD_{\sY,\bQ}\Hyp_{\pi}(\balpha;\bbeta)\longrightarrow
        \sH_{\pi}(\balpha;\bbeta)|_{\sY}
    \]
    of $\sD_{\sY,\bQ}$-modules.
\end{lemma}

\begin{proof}
    Since the leading coefficient $h_r(x)$ of $\Hyp_{\pi}(\balpha;\bbeta)$ is invertible on $\sX$
    by Lemma \ref{lem:coefficientofhgoperator},
    $\sD_{\sX,\bQ}/\sD_{\sX,\bQ}\Hyp_{\pi}(\balpha;\bbeta)$
    is a free $\sO_{\sX,\bQ}$-module of rank $r$;
    The previous lemma shows that the target of $\varphi$ is
    also a free $\sO_{\sY,\bQ}$-module of rank $r$.

    Therefore, it suffices to show that $\varphi$ is surjective.
    Since $\varphi$ is a morphism of coherent $\sO_{\sY,\bQ}$-modules,
    it suffices to prove that the morphism $\Gamma(\sY,\varphi)$ of the global sections is surjective.
    It follows because it has a dense image and because the source is complete.
\end{proof}

\begin{proof}[Proof of Theorem \ref{thm:ocFisoc}]
    Let $T'$ be a divisor of $\bP^1_k$ containing $T$ that satisfies
    the assumption in Lemma \ref{lem:phiisisom},
    which exists because $\sH_{\pi}(\balpha;\bbeta)|_{\sY}$ is an overholonomic $F$-$\sD^{\dag}_{\sY,\bQ}$-module
    \cite[Th\'eor\`eme 2.2.17]{Caro06CM}.
    Since $\sD_{\sY,\bQ}/\sD_{\sY,\bQ}\Hyp_{\pi}(\balpha;\bbeta)$ is the restriction
    of $\sp_{\plus}H_{\pi}(\balpha;\bbeta)$ to $\sY$,
    the isocrystal $H_{\pi}(\balpha;\bbeta)|_{\sY_K}$ is a convergent isocrystal on $Y$ \cite[Th\'eor\`eme 2.2.12]{Caro06CM}.
    Since the convergent property of an isocrystal can be checked by the restricting to a dense open subset \cite[Theorem 2.16]{Ogus84},
    $H_{\pi}(\balpha;\bbeta)$ itself is a convergent isocrystal.
    This shows that the $\sD_{\sX,\bQ}$-module structure on $\sp_{\plus}H_{\pi}(\balpha;\bbeta)$
    extends to a $\sD^{\dag}_{\sX,\bQ}$-module structure,
    and we obtain an isomorphism
    $\sp_{\plus}H_{\pi}(\balpha;\bbeta)\cong\sH_{\pi}(\balpha;\bbeta)|_{\sX}$
    of $\sD^{\dag}_{\sX,\bQ}$-modules.
    In particular, $\sH_{\pi}(\balpha;\bbeta)|_{\sX}$ is $\sO_{\sX,\bQ}$-coherent,
    which shows the claim \cite[Th\'eor\`eme 2.2.12]{Caro06CM}.
\end{proof}

\begin{remark}
    Let $\psi$ \resp{$\chi_i$'s and $\rho_j$'s} be a non-trivial character on $k$ \resp{characters on $k^{\times}$}
    corresponding to $\pi$ \resp{$\alpha_i$'s and $\beta_j$'s}.
    Then, Proposition \ref{prop:frobeniustrace} shows that the Frobenius trace of
    $H_{\pi}(\balpha;\bbeta)$ equals $(-1)^{m+n+1}\Hyp_{\psi}(\bchi;\brho)$
    (The difference from the Frobenius trace of $\sHyp_{\psi,!}(\bchi;\brho)$ comes from 
    the isomorphism of purity \cite[(1.2.14.1)]{AbeCaro}).
    As a result, the Frobenius trace of $H_{\pi}(\balpha;\bbeta)$ coincides with that of $\ell$-adic hypergeometric sheaf,
    which we denoted by $\sH_{\psi,!}^{\ell}(\bchi;\brho)$ in the proof of Lemma \ref{lem:dimofstalks}.
    \label{rem:explicittraces}
\end{remark}

\subsection{Irreducibility and a characterization of hypergeometric isocrystals.}

We firstly prove that the arithmetic hypergeometric $\sD$-modules are irreducible.

\begin{proposition}
    Let $\psi$ be a non-trivial additive character on $k$,
    and let $\chi_1,\dots,\chi_m$ and $\rho_1,\dots,\rho_n$ be characters on $k^{\times}$
    that satisfies $\chi_i\neq\rho_j$ for any $i,j$.
    Then, the following statements hold.

    {\rm (i)} $\sHyp_{\psi,\plus}(\bchi;\brho)$ is an irreducible object of $F\hyphen D^{\rb}_{\ovhol}(\bG_{\rmm,k}/K)$.

    {\rm (ii)} $j_{\plus}\sHyp_{\psi,\plus}(\bchi;\brho)$ is an irreducible object of $F\hyphen D^{\rb}_{\ovhol}(\bA^1_k/K)$.
    \label{prop:irredofD}
\end{proposition} 

\begin{proof}
    If $(m,n)=(0,0)$, then both assertions are trivial.
    We prove both (i) and (ii) by induction on $m+n$.
    In order to prove (i), 
    we first note that the irreducibility is preserved under taking $\inv^{\ast}$ and under the tensorisation of $\sK_{\gamma}$'s
    because these operations are equivalences of categories from $F\hyphen D^{\rb}_{\ovhol}(\bG_{\rmm,k}/K)$ to itself.
    Therefore, by the same argument as in the proof of Theorem \ref{thm:hypandconv},
    we may assume that $m>0$ and that $\chi_1$ is trivial.
    In the proof of Theorem \ref{thm:hypandconv}, we see that
    \[
        \sHyp_{\psi,\plus}(\bchi;\brho)\cong j^{\ast}\big(\FT_{\psi}(j_{\plus}\sHyp_{\pi_{\psi^{-1}},\plus}(\brho^{-1};\bchi'^{-1}))\big)
    \]
    if $\bchi'$ denotes the sequence $\chi_2,\dots,\chi_m$.
    By induction hypothesis, $j_{\plus}\sHyp_{\psi,\plus}(\brho^{-1};\bchi'^{-1})$
    is irreducible in $F\hyphen D^{\rb}_{\ovhol}(\bA^1_k/K)$.
    Because the geometric Fourier transform $\FT_{\psi}$ is an equivalence of categories,
    $\FT_{\psi}\big(j_{\plus}\sHyp_{\psi,\plus}(\brho^{-1};\bchi'^{-1})\big)$ is also irreducible.
    Because $j^{\ast}=j^!$ and because it preserves the property of being ``irreducible or zero'' \cite[1.4.6]{AbeCaro},
    the proof of (i) is done.
    To prove (ii), note that Corollary \ref{cor:middleext} shows that $j_{\plus}\sHyp_{\psi,\plus}(\bchi;\brho)$
    is the intermediate extension \cite[1.4.1]{AbeCaro} of $\sHyp_{\psi,\plus}(\bchi;\brho)$.
    Since the intermediate extension preserves the irreducibility \cite[1.4.7]{AbeCaro},
    (ii) is proved.
\end{proof}

This proposition also gives the irreducibility of hypergeometric isocrystals.

\begin{proposition}
    Let $\pi$ be an element of $K^{\times}$ satisfying $\pi^{q-1}=(-p)^{(q-1)/(p-1)}$,
    and let $\alpha_1,\dots,\alpha_m$ and $\beta_1,\dots,\beta_n$ be elements of $\frac{1}{q-1}\bZ$
    satisfying $\alpha_i-\beta_j\not\in\bZ$ for any $i,j$.
    If $(m,n)\neq (0,0)$, then the overconvergent $F$-isocrystal $H_{\pi}(\balpha;\bbeta)$
    on $X$ is irreducible, where $X$ denotes $\bG_{\rmm,k}$ if $m\neq n$ and $\bG_{\rmm,k}\setminus\{1\}$ if $m=n$.
\end{proposition}
\begin{proof}
    If $m\neq n$, then it follows from Proposition \ref{prop:irredofD}.
    If $m=n$, then it follows from the irreducibility of $k^{!}\sHyp_{\psi,\plus}(\bchi;\brho)$,
    where $k\colon \bG_{\rmm,k}\setminus\{1\}\hookrightarrow\bG_{\rmm,k}$ is the inclusion,
    which is a consequence of Proposition \ref{prop:irredofD} and \cite[Lemma 1.4.6]{AbeCaro}.
\end{proof}

At last, we give a characterization of hypergeometric isocrystals in terms of Frobenius trace function.

\begin{proposition}
    Let $\psi$ be a non-trivial additive character on $k$,
    and let $\chi_1,\dots,\chi_m$ and $\rho_1,\dots,\rho_n$ be characters on $k^{\times}$
    that satisfies $\chi_i\neq\rho_j$ for any $i,j$;
    let $\pi$ \resp{$\alpha_i$'s and $\beta_j$'s} be elements of $K$ \resp{$\frac{1}{q-1}\bZ$}
    corresponding to $\psi$ \resp{$\chi_i$'s and $\rho_j$'s}.
    Let $X$ be $\bG_{\rmm,k}$ if $m\neq n$ and be $\bG_{\rmm,k}\setminus\{1\}$ if $m=n$.

    Let $\iota\colon \overline{K}\cong\bC$ is an isomorphism.
    Let $\sF$ be an $\iota$-mixed overconvergent $F$-isocrystal on $X$
    whose Frobenius trace at each closed point $x$ equals $(-1)^{m+n+1}\Hyp_{\psi}(\bchi;\brho)$.
    Then, there exists an isomorphism $\sF\to H_{\pi}(\balpha;\bbeta)$ of overconvergent $F$-isocrystals.
    Moreover, we have
    \[\dim_K\Hom(\sF, H_{\pi}(\balpha;\bbeta))=1.\]
\end{proposition}

\begin{proof}
    First, note that $H_{\pi}(\balpha;\bbeta)$ is pointwise $\iota$-pure
    because the Frobenius traces coincide with those of a smooth $\ell$-adic sheaf
    $\sH^{\ell}_{\psi}(\bchi;\brho)$ and the latter is so \cite[Theorem 8.4.2 (4)]{Katz90ESDE}.
    Therefore, by the ``\v{C}ebotarev density theorem'' for overconvergent $F$-isocrystals \cite[3]{Abe11},
    the semi-simplifications of $\sF$ and $H_{\pi}(\balpha;\bbeta)$ are isomorphic to each other.
    Because $H_{\pi}(\balpha;\bbeta)$ is irreducible by the previous proposition,
    $\sF$ is itself isomorphic to $H_{\pi}(\balpha;\bbeta)$.
    In particular, $\sF$ is irreducible, which shows the last part of the proposition.
\end{proof}

\end{document}